\documentclass[a4paper, 10pt]{amsart}

\usepackage{amsmath,amsfonts,amssymb}
\usepackage{url}
\usepackage[pagebackref]{hyperref}
\usepackage{comment}
\usepackage{multicol}
\usepackage[noabbrev]{cleveref} 

\newtheorem{thm}{Theorem}[section]
\newtheorem{cor}[thm]{Corollary}
\newtheorem{lem}[thm]{Lemma}
\newtheorem{prop}[thm]{Proposition}
\theoremstyle{remark}
\newtheorem{defn}{Definition}
\newtheorem{rmk}{Remark}

\def\Z{{\mathbb Z}}
\def\Q{{\mathbb Q}}
\def\R{{\mathbb R}}
\def\C{{\mathbb C}}
\def\T{{\mathbb T}}
\def\H{{\mathbb H}}
\def\A{{\mathbb A}}
\def\OK{{\mathcal O}_K}
\def\RC{{\widehat{K}}}
\def\P{{\mathbb P}}

\def\M{{\mathcal M}}
\def\Hyp{{\mathcal H}}
\def\SS{{\mathcal S}}
\def\G{{\mathcal G}}

\def\J{{\mathcal J}}
\def\PP{{\mathcal P}}

\def\a{{\mathfrak a}}
\def\b{{\mathfrak b}}
\def\ab{(\a,\b)}
\def\Xab{X_{\a,\b}}

\def\d{{\mathfrak d}}

\def\m{{\mathfrak m}}
\def\n{{\mathfrak n}}
\def\p{{\mathfrak p}}
\def\q{{\mathfrak q}}

\def\Go{\Gamma_0}
\def\Gon{\Gamma_0(\n)}
\def\Gi{\Gamma_1}
\def\Gin{\Gamma_1(\n)}

\def\RR{{\OK\oplus\OK}}
\def\KK{{K\oplus K}}

\DeclareMathOperator{\diag}{diag}
\DeclareMathOperator{\adj}{adj}
\DeclareMathOperator{\Mat}{Mat}

\DeclareMathOperator{\GL}{GL}
\DeclareMathOperator{\SL}{SL}

\DeclareMathOperator{\SU}{SU}

\DeclareMathOperator{\SO}{SO}
\DeclareMathOperator{\Cl}{Cl}

\def\<#1>{\left<#1\right>}
\newcommand{\mat}[4]{\left(\begin{matrix} %
                   #1 & #2 \\ #3 & #4 %
                  \end{matrix}\right)}

\def\Magma{{\sc Magma}}

\begin{document}

\title[Hecke operators and Modular Forms over Number Fields]{Hecke operators, Hecke Eigensystems, and Formal Modular Forms over Number Fields}

\author{J. E. Cremona}
\address{
Mathematics Institute, University of
Warwick, Coventry CV4 7AL, UK.
}
\email{J.E.Cremona@warwick.ac.uk}
\date{\today}

\begin{abstract}
  We develop an explicit theory of formal modular forms over arbitrary
  number fields~$K$, as functions of modular points.  We define
  modular points for $\Go(\n)$ and~$\Gi(\n)$, where the level~$\n$ is
  an integral ideal of~$K$; Hecke operators and generalized
  Atkin-Lehner operators as functions of modular points; and
  associated Hecke eigensystems.  We show how complete eigensystems
  may be recovered, uniquely up to unramified quadratic twist, from
  their restrictions to principal Hecke operators, and we give
  explicit formulas for principal operators suitable for machine
  computation. These have been implemented by the author in the case
  of imaginary quadratic fields, and used in his systematic
  computation of Bianchi cusp forms, which are available in the
  L-functions and modular forms database (LMFDB).  While our
  description incorporates the classical theory for $K=\Q$, and also
  extends work of the author and his students for imaginary quadratic
  fields, it applies to arbitrary number fields, and may be useful
  in the computation of spaces of automorphic forms for $\GL(2,K)$
  over number fields, whether via modular symbols or other
  methods.
\end{abstract}

\maketitle

\section{Introduction} \label{sec:intro}
Let $K$ be an arbitrary number field with ring of integers~$\OK$.
Spaces of cusp forms of weight~$2$ (and higher weights) for $\GL(2,K)$
have been computed for certain classes of fields~$K$ using methods
based on modular symbols.  For $K=\Q$, such computations are
classical: see \cite{JCbook2} and \cite{Stein-book}.  For imaginary
quadratic fields, the methods have been developed by the author and
his students: see \cite{JChyp}, the author's thesis~\cite{JCthesis},
and the theses of Whitley \cite{Whitley}, Bygott \cite{JBthesis}, and Lingham
\cite{LinghamThesis}.  For some real quadratic fields, explicit
computations have been carried out by students of Darmon, notably
Demb\'el\'e \cite{DembeleThesis} for $\Q(\sqrt{5})$; using different
methods, extensive computations of Hilbert modular forms have been
made by Voight and others for totally real fields of degree up
to~$6$. Lastly, Gunnells and Yasaki have computed cusp forms over two
fields of mixed signature, namely the cubic field of
discriminant~$-23$ and~$\Q(\zeta_5)$.  In each case, one is interested
in explicitly computing the action of the algebra of Hecke operators
on the space and determining the associated eigensystems.

In this paper we give a systematic treatment of those parts of the
theory which are purely algebraic in nature, and apply to all number
fields.  For much of what follows, we need only assume that $\OK$ is a
Dedekind domain; special properties enjoyed by the ring of integers of
a number field, such as finiteness of the class group and of residue
fields, will only be used where necessary.

We will develop an explicit theory of \emph{formal modular forms} over
the field~$K$, as functions of \emph{modular points}.  We will define:
modular points for $\Go(\n)$ and~$\Gi(\n)$, where the level~$\n$ is an
integral ideal of~$K$; Hecke operators and generalized Atkin-Lehner
operators as functions of modular points; and associated Hecke
eigensystems.  We will show (see~\Cref{thm:eig-systems}) how complete
eigensystems may be recovered, uniquely up to unramified quadratic
twist, from their restrictions to principal Hecke operators, and we
give explicit formulas for principal operators suitable for machine
computation. These have been implemented by the author in the case of
imaginary quadratic fields, and used in his systematic computation of
Bianchi cusp forms, which are available\footnote{To date (2026), only
Bianchi newforms with rational Hecke eigenvalues are in the LMFDB.} in
the L-functions and modular forms database (LMFDB).  While our
description incorporates the classical theory for $K=\Q$, and also
extends work of the author and his students for imaginary quadratic
fields, it may be useful more generally in the computation of
spaces of automorphic forms for $\GL(2,K)$ over other number fields,
whether via modular symbols or other methods.

\subsection{Outline of the paper.}

In the next section, we recall some basic concepts and notation
introduced by the author and Aran\'es in~\cite{JC+MA} concerning
congruence subgroups, $\ab$-matrices, and M-symbols.  In
\Cref{sec:lattices} we introduce lattices and modular points for a
general number field~$K$, with respect to the congruence
subgroups~$\Go(\n)$ and~$\Gi(\n)$, for an arbitrary level~$\n$. In
\Cref{sec:Hecke}, we introduce Hecke operators as operators on the
free module generated by modular points, and define the grading of the
Hecke algebra by the class group. Hecke eigensystems are the subject
of~\Cref{sec:eigenvalue-systems}.  A key result
(\Cref{thm:eig-systems}) is the determination of the extent to which
such an eigensystem is determined by its restriction to the principal
subalgebra of the Hecke algebra.  This has important practical
consequences, since in explicit computations, such as using modular
symbols, computing principal Hecke operators is simpler than the
general case; our result implies that \emph{all} information about
Hecke eigensystems may be recovered from the eigenvalues of principal
operators.  The final~\Cref{sec:hecke-matrices} makes explicit the
matrices defining principal Hecke and Atkin-Lehner operators and
principal combinations of these.  These have been implemented by the
author, and used in practice, over imaginary quadratic fields, to
determine Hecke eigensystems over arbitrary imaginary quadratic
fields, in his {\tt C++} package~{\tt bianchi-progs}
\cite{bianchi-progs}, and also in \Magma\ by K.~Thalagoda; see her
thesis~\cite{KalaniThesis}, and our forthcoming paper~\cite{CTY}
(joint with Thalagoda and Yasaki), for more details of this
application.

\section{Notation and basic definitions}
\label{sec:basic}
Let $K$ be an arbitrary number field, with ring of integers~$\OK$, and
unit group~$\OK^\times$.  Let $\Mat_2(K)$ and $\Mat_2(\OK )$ denote
the algebras of~$2\times2$ matrices with entries in~$K$ or~$\OK$
respectively, and  $\GL(2,K)$ and $\Gamma:=\GL(2,\OK)$ the
corresponding multiplicative groups.

Nonzero ideals of~$\OK$ are denoted $\a, \b, \dots, \n$ and prime
ideals by~$\p, \q$.  The norm of an ideal is $N(\a)=\#(\OK/\a)$,
satisfying $N(\a\b)=N(\a)N(\b)$ for all~$\a,\b$.  We have
$\#(\OK/\a)^\times = \varphi(\a)$, where
\[
   \varphi(\a) := N(\a)\prod_{\p\mid\a}\left(1-N(\p)^{-1}\right).
\]

Associated to each nonzero integral ideal $\n$ of~$\OK$, called the
\emph{level}, we have the standard congruence subgroups\footnote{Our
base group is $\Gamma=\GL(2,\OK)$, not~$\SL(2,\OK)$, so we define
$\Gon$ and $\Gon$ accordingly.} $\Gon$, $\Gin$ and~$\Gamma(\n)$ of
$\Gamma$.  We will only be concerned with $\Gon$ and~$\Gin$ here:
\begin{align*}
\Gon =
\left\{\mat{a}{b}{c}{d}\in\Gamma\mid c\in\n\right\};
\\
\Gin =
\left\{\mat{a}{b}{c}{d}\in\Gamma\mid c, d-1\in\n\right\}.
\end{align*}
We have $[\Gamma:\Gon] = \psi(\n)$, where
\[
   \psi(\n) := N(\n)\prod_{\p\mid\n}\left(1+N(\p)^{-1}\right),
\]
and $[\Gon:\Gin] = \varphi(\n)$.  Note that $\varphi$
and~$\psi$ are multiplicative in the sense
that~$\varphi(\m\n)=\varphi(\m)\varphi(\n)$ when $\m,\n$ are coprime,
and similarly for~$\psi$.  This, and the index formulas, are proved
exactly as in the case $K=\Q$.

The group of fractional ideals coprime to the level~$\n$ is
denoted~$\J_K^{\n}$.

For a level~$\n$, divisors $\q$ of~$\n$ such that $\q$ and~$\n\q^{-1}$
are coprime are called \emph{exact divisors} of~$\n$; this relation is
indicated by~$\q\mid\mid\n$.

We will use standard facts about finitely-generated modules over the
Dedekind Domain~$\OK$.  Specifically we are concerned with
\emph{$\OK$-lattices of rank~$2$}, which are $\OK$-submodules of $\KK$
of full rank. (Below we will also consider more
general~$\OK$-lattices.)  Elements of~$\OK$-lattices are written as
row vectors, with matrices acting on the right.  For both the
theoretical results we will need and an algorithmic approach to this
theory as well as solutions to many computational problems which arise
in practice, see Chapter~1 of Cohen's book~\cite{Cohen2}.

The group~$\Gamma$ may be characterized through its action on
$\OK$-lattices, as the set of all matrices~$\gamma\in\GL(2,K)$
satisfying $(\RR)\gamma=\RR$, and the following similar
characterization of~$\Gon$ is easily checked:
\begin{align*}
\Gon & = \{\gamma\in\Gamma\mid (\n\oplus \OK)\gamma=(\n\oplus \OK)\}\\ & =
\{\gamma\in\GL(2,K)\mid (\RR)\gamma=\RR \ \text{and}\ (\n\oplus
\OK)\gamma=(\n\oplus \OK)\}.
\end{align*}
Here, we could also replace~$\n\oplus \OK$ by~$\OK\oplus\n^{-1}$.
Thus $\Gon$ is the right stabilizer of the pair of lattices $(L,L')$
where $L=\RR$ and $L'=\OK\oplus\n^{-1}$, so that $L'\supseteq L$ and
$L'/L\cong \OK/\n$ as $\OK$-modules.  The pair~$(L,L')$ is an example
of a \emph{modular point} for $\Gon$ which will be studied in detail
in~\Cref{sec:lattices} below.

Alternatively, in terms of the
subalgebra~$A_0(\n)=\mat{\OK}{\OK}{\n}{\OK}$ of~$\Mat_2(\OK)$, we see
that $\Gon$ is the intersection of this with~$\Gamma$, consisting of
those matrices~$\gamma$ such that~$A_0(\n)\gamma=A_0(\n)$.

To characterize~$\Gin$ in the same way, we need to rigidify the
lattice pair~$(L,L')$ by fixing an $\OK$-module generator of the
quotient~$L'/L$ (which is a cyclic $\OK$-module, being isomorphic
to~$\OK/\n$).  To do this, let $n_0\in\n^{-1}$ generate $\n^{-1}/\OK$
(so that $\n^{-1}=\OK+n_0\OK$) and set $\beta_0=(0,n_0)\in L'$.  Then
$L'=L+\OK\beta_0$, and $\Gin$ is the subgroup of~$\Gamma$
fixing~$\beta_0\pmod{L}$:
\begin{align*}
\Gin  =
  \{\gamma\in\GL(2,K)\mid & (\RR)\gamma=\RR, \quad\text{and}\\
  & \beta_0\gamma=\beta_0\pmod{\OK\oplus \OK}\}
\end{align*}
since for $x\in \OK$ we have $n_0x\in \OK\iff x\in\n$.  This
characterization is independent of the choice of~$n_0$.  When $\OK=\Z$
and $\n=N\Z$, one usually takes $n_0=\frac{1}{N}$, but there is no
canonical choice in general.

\smallskip

For nonzero $a,b\in \OK$ (or~$K$) we denote by~$\<a>$ and $\<a,b>$ the
(fractional) ideals $a\OK$ and $a\OK+b\OK$.  The class of a fractional
ideal~$\a$ in the class group $\Cl(K)$ is denoted by~$[\a]$. The order
of the class group~$\Cl(K)$ is denoted~$h(\OK)$ or simply~$h$.

Among the elementary properties of Dedekind Domains we will use are
\begin{enumerate}
\item For any ideal~$\a$ and any nonzero $a\in\a$ there exists
  $b\in\a$ with $\a=\<a,b>$;
\item Every ideal class contains an ideal coprime to any given ideal.
\end{enumerate}

\subsection{\texorpdfstring{$\ab$-matrices}{(a,b)-matrices}} 
\label{sec:ab-mats}

Certain matrices in $\Mat_2(\OK)$ called $\ab$-matrices were
introduced in \cite{JC+MA} and will play an important role in our theory.
Over~$\Q$, or more generally when the class number is~$1$, these
matrices may not be explicitly visible, as they function like matrices in Hermite
Normal Form.  In \cite{JC+MA}, they were used to study and count the
orbits of~$\Gon$ and~$\Gin$ on $\P^1(K)$.

$\ab$-matrices are associated with any pair of ideals~$\a,\b$ in
inverse ideal classes, so that~$\a\b$ is principal, and
$\a\oplus\b\cong \RR$.  An~$\ab$-matrix is any matrix
$M\in\Mat_2(\OK)$ such that
\[
   (\RR)M=\a\oplus\b.
\]
The set of all~$\ab$-matrices is denoted~$\Xab$, and for all pairs
of integral ideals~$\ab$ we set
\begin{align}
  \label{eqn:Delta-ab-defn}
\Delta(\a,\b) &= \left\{\mat{x}{y}{z}{w}\mid x,w\in
\OK, y\in\a^{-1}\b, z\in\a\b^{-1}, xw-yz\in \OK^\times\right\}\\
&=\left\{\gamma\in \GL(2,K)\mid(\a\oplus\b)\gamma=\a\oplus\b\right\},\nonumber
\end{align}
by Proposition~2.4 of~\cite{JC+MA}.

In the case $\a=\b=\OK$, an $\ab$-matrix is simply an element
of~$\Gamma$.  From~\cite{JC+MA} we recall the following facts
about~$\ab$-matrices:
\begin{enumerate}
\item
$\ab$-matrices exist whose first column is an arbitrary pair of
  generators of~$\a$.  In particular, the lower left entry may be
  chosen to be any nonzero element of~$\a$, so may be taken to lie in
  any other ideal~$\n$ that we choose.  An~$\ab$-matrix whose
  $(2,1)$-entry lies in~$\n$ is called an~\emph{$\ab$-matrix of
  level~$\n$}.
\item \( \Xab = \Gamma M = M\Delta\ab\) for all~$M\in \Xab$.
\end{enumerate}

\subsection{\texorpdfstring{M-symbols and coset representatives for  $\Go(\n)$}
  {M-symbols and coset representatives}}
\label{sec:M-symbols}
As in Section~3 of~\cite{JC+MA}, for an ideal~$\n$ of~$\OK$, an
\emph{M-symbol} (or \emph{Manin symbol}) \emph{of level~$\n$} is an
element $(c:d)\in\P^1(\OK/\n)$.  Here, $c,d\in\OK$ are such that the
ideal~$\<c,d>$ is coprime to~$\n$, and
\[
(c_1:d_1) = (c_2:d_2) \iff c_1d_2-c_2d_1\in\n.
\]
For brevity we will write $\P^1(\n)$ for $\P^1(\OK/\n)$.  The
M-symbols of level~$\n$ are in bijection with the set of cosets
of~$\Go(\n)$ in~$\Gamma$, since
\[
\mat{a_1}{b_1}{c_1}{d_1}\Go(\n) =
\mat{a_2}{b_2}{c_2}{d_2}\Go(\n) \iff (c_1:d_1) = (c_2:d_2)
\quad\text{in $\P^1(\n)$}.
\]
In other words, the map $\Gamma \to \P^1(\n)$ is surjective, with
fibres equal to the cosets of~$\Go(\n)$.

Since the reduction map~$\SL(2,\OK)\to\SL(2,\OK/\n)$ is
surjective\footnote{Note, however, that the
map~$\Gamma=\GL(2,\OK)\to\GL(2,\OK/\n)$ is not in general surjective,
since $\OK^*\to(\OK/\n)^*$ is not.}, for all~$c,d\in\OK$ such that
$\<c,d>+\n=\OK$ there exist~$c',d'\in\OK$ such that $\<c',d'>=\OK$ and
$(c:d)=(c':d')$. Writing $1=ad'-bc'$ with~$a,b\in\OK$, we obtain a
section from~$\P^1(\n)$ to~$\Gamma$, the implementation of which is
crucial when doing computations.

More generally, suppose that the ideals~$\m,\n$ are coprime.  Then
M-symbols of level~$\n$ also give coset representatives
of~$\Go(\m\n)$ in~$\Go(\m)$.  To see this, note that by the
Chinese Remainder Theorem we may identify~$\P^1(\m\n)$ with
$\P^1(\m)\times\P^1(\n)$.  Hence, for each~$(c:d)\in\P^1(\n)$, we may
find $c',d'\in\OK$ which are coprime and such that
\[
(c':d')=(c:d)
\quad
\text{in $\P^1(\n)$, and}
\quad
(c':d')=(0:1)
\quad
\text{in $\P^1(\m)$}.
\]
Then, lifting the M-symbol $(c':d')\in\P^1(\m\n)$ to~$\Gamma$ as
before results in a matrix in~$\Go(\m)$ representing the same
$\Go(\n)$-coset as $(c:d)$.

The algorithmic procedures outlined here play an important part in the
construction of explicit matrices giving the action of Hecke and related
operators below.

\section{Lattices and Modular points} \label{sec:lattices}

One description of classical modular forms for $\SL(2,\Z)$ is to view
them as functions on classical lattices---discrete rank two
$\Z$-submodules of $\C$---satisfying certain properties.  This approach
may be extended to modular forms for $\Go(N)$ or $\Gi(N)$ by
giving the lattices extra ``level~$N$'' structure.  For example this
approach is taken in the books by Serre, Koblitz and Lang: see
\cite[VIII~\S2]{SerreCA} for the case of level~$1$ and
\cite[III~\S5]{KoblitzECMF} and \cite[VII]{LangModularForms} for
$\Go(N)$ and $\Gi(N)$.  This description of modular forms
enables one to define Hecke operators via their action on lattices.

In his unpublished thesis \cite[Chap.~7]{JBthesis}, Bygott began to extend this
approach to number fields, the emphasis there being on the imaginary
quadratic case.  We follow and extend that approach here.  The first
step is to consider $\OK$-lattices more general than $\OK$-submodules
of $\KK$; then we introduce the level structure.  We start with the
case of $\Go(N)$-structure, and then consider~$\Gi(N)$.

Let $\RC=K\otimes_{\Q}\R\cong\R^{r_1}\oplus\C^{r_2}$, where $r_1$
and~$r_2$ are numbers of real and complex conjugate pairs of
embeddings of~$K$ into its Archimedean completions, $\R$ or~$\C$.  The
components of~$\RC$ will be denoted $K_{\infty,r}$ for $1\le r\le
r_1+r_2$, and the real and complex embeddings of~$K$ give a diagonal
embedding~$K\hookrightarrow\RC$.

Define $\G=\GL(2,\RC)$; this group acts on the right, component-wise,
on the space of row vectors~$\RC^2=\RC\oplus\RC$.

\begin{defn}
An $\OK$-lattice~$L \subset \RC^2$ is an $\OK$-submodule of~$\RC^2$
satisfying all of the following conditions:
\begin{enumerate}
\item $\RC L = \RC^2$;
\item $F_1\subseteq L\subseteq F_2$ where $F_1$, $F_2$ are free
  $\OK$-modules of rank~$2$;
\item $L$ is finitely-generated and $\dim_K(KL)=2$.
\end{enumerate}
\end{defn}
(In fact, conditions (2) and~(3) are equivalent.)  So
$E=KL$ is a $2$-dimensional $K$-vector subspace of~$\RC^2$, hence
isomorphic to~$\KK$, but in general different lattices lie in
different such subspaces.  Each lattice~$L$ has $r_1+r_2$ components,
the $r$th component contained in $K_{\infty,r}^2$, which is either
$\R\oplus\R$ or~$\C\oplus\C$.  It is sometimes convenient to
identify these with~$\C=\R+\R i$ and $\H=\C+ \C j$ (the algebra of
Hamiltonian quaternions) respectively.

The following elementary facts about lattices are stated without
proof, as all follow easily from the theory of finitely-generated
modules over Dedekind Domains.

Every lattice~$L$ is isomorphic as $\OK$-module to $\OK\oplus\a$ for
some ideal~$\a$, whose class in the class group~$\Cl(K)$ is the
(Steinitz) \emph{class} of~$L$, which we denote~$[L]$.  The structure
theorem for $\OK$-modules implies that two lattices are isomorphic as
$\OK$-modules if and only if they have the same class, since they are
both torsion-free and of rank~$2$.

If $L$ is a lattice and $U\in\G$, then $LU$ is also a lattice,
isomorphic to~$L$ and hence of the same class.  Conversely, any
$\OK$-module isomorphism between lattices $L_1\to L_2$ extends first
to a $K$-linear isomorphism $KL_1\to KL_2$, and then to a $\RC$-linear
automorphism of~$\RC^2$, so is represented by a
matrix~$U\in\G$.  Hence
\[
[L_1] = [L_2] \iff L_1 \cong L_2 \iff L_2=L_1U\qquad\text{for some
  $U\in\G$}.
\]
One may also think of $L$ and $LU$ as being the same lattice,
expressed in terms of different bases for~$\RC^2$, with $U$ being the
change of basis matrix.

In the real case, when~$\RC=\R$ and $\G=\GL(2,\R)$, we will always use
an \emph{oriented basis}~$v_1,v_2\in\R^2$ for each lattice~$L$,
meaning that the matrix~$U$ with rows~$v_1,v_2$ has positive
determinant. Any basis can be changed into an oriented basis by
negating one of the basis vectors, if necessary.

If $\a$ is a fractional ideal of~$\OK$, then the module $\a L$,
consisting of all finite sums~$\{\sum_ia_ix_i\mid a_i\in\a, x_i\in
L\}$, is also a lattice, of class $[\a]^2[L]$.  For example, if $[L]$
is a square, say $[L]=[\a]^2$, then $L\cong\a\oplus\a$, so
$\a^{-1}L\cong\OK\oplus\OK$, a free lattice.

If $L$, $L'$ are two lattices with $L'\supseteq L$ then $L'/L$ is a
finite torsion $\OK$-module, hence is isomorphic to $\OK/\a_1\oplus
\OK/\a_2$ for unique ideals $\a_1$, $\a_2$ of~$\OK$ satisfying
$\a_1\mid\a_2$, called the \emph{elementary divisors} of~$L'/L$.  The
product $\a=\a_1\a_2$ is called the \emph{norm ideal} of $L'/L$ or the
\emph{index ideal} of~$L$ in~$L'$, denoted $[L':L]$; it determines the
cardinality $\#(L'/L)=N(\a)=\#(\OK/\a)$, and $\a_2$ is the annihilator
of $L'/L$.  The classes are related by $[L']=[\a][L]$.  For any
fractional ideal~$\b$ we have $\b L'/\b L\cong L'/L$.  Moreover, there
exist fractional ideals $\b_1,\b_2$ and a $\RC$-basis $v_1,v_2$
of~$\RC^2$ such that
\[
     L'=\b_1v_1\oplus\b_2v_2; \qquad L=\a_1\b_1v_1\oplus\a_2\b_2v_2;
\]
we can write this alternatively as
\[
     L'=(\b_1\oplus\b_2)U; \qquad L=(\a_1\b_1\oplus\a_2\b_2)U
\]
where $U\in\G$ is the matrix with rows~$v_1,v_2$.  By scaling~$v_1$
and~$v_2$, we may replace $\b_1,\b_2$ by integral ideals.  Now we have
$L'/L\cong \OK/\a_1\oplus \OK/\a_2$.

\subsection{Standard Lattices}
We wish now to choose a standard lattice in each Steinitz class.  One
way to do this (as in Bygott's thesis \cite{JBthesis}), which works
for any field, is to take ideals $\p_i$ representing the ideal
classes, such as $\p_1=\OK$ and $\p_i$ prime for $i>1$, and use the
lattices $\p_i\oplus \OK$.  When working with level~$\n$ structure, we
may also wish to assume that the $\p_i$ are coprime to~$\n$.  However,
there is another choice available for lattices whose class is
square, since a class of the form $[\q]^2$ may be represented by the
lattice~$\q\oplus\q=\q(\RR)$.  Again, we may wish to choose $\q$ to be
coprime to the current level~$\n$.  When the class number is odd, we
can use this second choice for all ideal classes, resulting in a
simplification of the theory in several places.

In the general case, we combine these two approaches.  Write
$h=\#\Cl(K)=h_2h_2'$ where $h_2=\#\Cl[2]=[\Cl(K):\Cl(K)^2]$.  For
$1\le i\le h_2$ we fix ideals~$\p_i$ whose classes~$[\p_i]$ represent
the cosets $\Cl(K)/\Cl(K)^2$, and for $1\le j\le h_2'$ we fix~$\q_j$
whose classes~$[\q_j]^2$ comprise~$\Cl(K)^2$.  We always take
$\p_1=\q_1=\OK$; the other representatives may if convenient be taken
to be primes, or coprime to a given level, or both.  Then the set of
all ideal classes is
\[
  \Cl(K) = \{c_{ij}=[\p_i][\q_j]^2\mid 1\le i\le h_2, 1\le j\le h_2'\}.
\]
When the class number is odd we have $h_2=1$ and then all
representatives are of the form~$[\q_j]^2$; at the other extreme, if
$\Cl(K)$ has exponent~$2$ then $h_2=h$ and all are of the
form~$[\p_i]$.

For each class~$c_{ij}\in\Cl(K)$, we define the \emph{standard
lattice in class~$c_{ij}$} to be
\[
   L_{ij} = \p_i\q_j\oplus\q_j = \q_j(\p_i\oplus \OK).
\]
Since every lattice is isomorphic to $L_{ij}$ for a unique
pair~$(i,j)$, the set of lattices in class~$c_{ij}$ is $\{L_{ij}U\mid
U\in\G\}$.  The matrix $U$ in these representations is of course not
unique, as we now make precise.

As a special case of~\ref{eqn:Delta-ab-defn} we set
\[
\Gamma_\p=\Delta(\p,\OK)=\left\{\mat{a}{b}{c}{d}\mid a,d\in \OK,
b\in\p^{-1},c\in\p, ad-bc\in \OK^\times\right\}.
\]
For brevity, we set $\Gamma_i=\Gamma_{\p_i}$.  Note that if
$\p=\<\pi>$ were principal then $\Gamma_\p$ would be simply the
conjugate of~$\Gamma$ by~$\mat{\pi}{0}{0}{1}^{-1}$.  The following
result follows from Proposition~2.4 of~\cite{JC+MA}, after observing
that for any lattice~$L$, if $LU=L$ for some~$U\in\G$, then the
entries of~$U$ lie in~$K$.

\begin{prop}
Let $U\in\G$.  Then $L_{ij}U=L_{ij}$ if and only if $U\in
\Gamma_i$ (independent of~$j$); hence
\[
  L_{ij}U = L_{ij}U' \iff U'U^{-1}\in\Gamma_i.
\]
The set of lattices in class~$c_{ij}$ is in bijection with the coset
space~$\Gamma_i\backslash\G$.\qed
\end{prop}

Hence when the class number is odd, each class of lattices is
parameterised by the \emph{same} space~$\Gamma\backslash\G$,
just as in the case of trivial class group.

\subsection{Modular points for \texorpdfstring{$\Gon$}{Gamma0(n)}}
\begin{defn}
A \emph{modular point for $\Gon$} is a pair $(L,L')$ of
lattices\footnote{In \cite{JBthesis}, the pair $(L,S)$ where $S=L'/L$
is called a modular point for $\Gon$ rather than $(L,L')$, but this is
only a cosmetic difference.} such that $L'\supseteq L$ and $L'/L\cong
\OK/\n$.  We call $L$ the \emph{underlying lattice} of the pair, and
$[L]$ its \emph{class}.  The set~$\M_0(\n)$ of modular points for
$\Gon$ is the disjoint union of the sets~$\M_0^{(c)}(\n)$ of modular
points in each ideal class~$c$.
\end{defn}

Thus a modular point consists of an underlying lattice and a
superlattice (note, not a sublattice) with relative index ideal~$\n$,
such that the quotient is a cyclic $\OK$-module.  If
$(L,L')\in\M_0(\n)$ then $L\subseteq L'\subseteq\n^{-1}L$.  Modular
points certainly exist with any underlying lattice~$L$, therefore,
since $\n^{-1}L/L\cong(\OK/\n)^2$ has submodules isomorphic
to~$\OK/\n$.

If $P=(L,L')\in\M_0(\n)$ then $\a P=(\a L,\a L')\in\M_0(\n)$ also for
every ideal~$\a$, since $\a L'/\a L\cong L'/L$.  Also,
$PU=(LU,L'U)\in\M_0(\n)$ for all~$U\in\G$.

The elementary divisors of~$L'/L$ are $\a_1=\OK$ and~$\a_2=\n$, so
there exist fractional ideals~$\b_1,\b_2$ and $U\in\G$ such that
\[
     L'=(\b_1\oplus\b_2)U; \qquad L=(\b_1\oplus\n\b_2)U.
\]
It is now convenient to change this notation, replacing $\b_2$
by~$\n\b_2$, so that now the modular point $(L,L')$ has the form
\[
     L=(\b_1\oplus\b_2)U; \qquad L'=(\b_1\oplus\n^{-1}\b_2)U.
\]

The basic example of a modular point is the
pair~$(\RR,\OK\oplus\n^{-1})$ with underlying lattice~$\RR$.  More
generally, to each of our standard lattices~$L_{ij}$ we associate a
\emph{standard modular point}~$P_{ij}$:
\[
   P_{ij}=(L_{ij},L_{ij}') =
   (\p_i\q_j\oplus\q_j,\p_i\q_j\oplus\q_j\n^{-1})  = \q_j(\p_i\oplus
   \OK,\p_i\oplus\n^{-1}).
\]
The following proposition states that every modular point is the image
of a standard modular point under some $U\in\G$.

\begin{defn}
An \emph{admissible basis matrix} for the modular
point~$P=(L,L')\in\M_0^{(c_{ij})}(\n)$ is a matrix $U\in\G$ such that
$P=P_{ij}U$; that is, such that
\[
   L=\q_j(\p_i\oplus \OK)U; \qquad\qquad L'=\q_j(\p_i\oplus\n^{-1})U.
\]
The rows of an admissible basis matrix are vectors~$v_1, v_2\in\RC^2$,
which we call an \emph{admissible basis} for the modular point~$P$;
they are characterized by
\[
   L = \q_j(\p_i v_1+\OK v_2);\qquad\qquad
   L'= \q_j(\p_i v_1+\n^{-1}v_2).
\]
\end{defn}

We now show that admissible bases exist, and see to what extent they
are unique.

\begin{prop} \label{prop:ad-basis-exist-0}
Every modular point for~$\Gon$ has an admissible basis.
\end{prop}
\begin{proof}
First consider the case of a principal modular point $(L,L')$, where
$L$ is free.  Write $L=(\b_1\oplus\b_2)U_0$ and
$L'=(\b_1\oplus\n^{-1}\b_2)U_0$ as above, where $\b_1,\b_2$ are
integral ideals and $U_0\in\G$.  Since $L$ is free, $\b_1\b_2$
is principal; let $V\in\GL(2,K)$ be a $(\b_1,\b_2)$-matrix of
level~$\n$.  Then $(\RR)V=\b_1\oplus\b_2$, so that putting
$U=VU_0\in\G$ we have $(\RR)U=(\b_1\oplus\b_2)U_0=L$;
moreover, since $V$ has level~$\n$, Proposition~2.6 of~\cite{JC+MA}
implies that $(\OK\oplus\n^{-1})V=\b_1\oplus\n^{-1}\b_2$, and hence
$(\OK\oplus\n^{-1})U=(\b_1\oplus\n^{-1}\b_2)U_0=L'$.  Thus, $U$ is an
admissible basis matrix for~$(L,L')$.

Next we show how to construct an admissible basis matrix in the case
where $L\cong L_{i1}=\p_i\oplus \OK$ with~$i>1$.  Again we start from a
representation of the form~$L=(\b_1\oplus\b_2)U_0$ and
$L'=(\b_1\oplus\n^{-1}\b_2)U_0$ with $U_0\in\G$.  We claim the
existence of $V\in\mat{\b_2^{-1}}{\b_1^{-1}}{\b_1\n}{\b_2}$ with~$\det
V=1$.  Then~$V$ satisfies $(\b_1\b_2\oplus \OK)V=\b_1\oplus\b_2$ and
$(\b_1\b_2\oplus\n^{-1})V = \b_1\oplus\n^{-1}\b_2$.  Finally, since
$L\cong L_{i1}$, we have $\b_1\b_2=t\p_i$ for some~$t\in K^*$, so if
we set $U=\mat{t}{0}{0}{1}VU_0$ then $U$ is an admissible basis matrix
for $(L,L')$:
\begin{align*}
  (\p_i\oplus \OK)U &= (\b_1\b_2\oplus \OK)VU_0=(\b_1\oplus\b_2)U_0=L;\\
  (\p_i\oplus\n^{-1})U &=
  (\b_1\b_2\oplus\n^{-1})VU_0=(\b_1\oplus\n^{-1}\b_2)U_0=L'.\\ 
\end{align*}
For the existence of~$V$ we proceed as follows, generalizing
\cite[Prop.~1.3.12]{Cohen2} where the special case $\n=\OK$ is proved.
Choose an integral ideal~$\a$ coprime to~$\b_2$ in the inverse ideal
class to~$\n\b_1$.  Then $\a\n\b_1=\<z>$ with $z\in\n\b_1$.  Next
choose $x$ so that $x\b_2$ is integral and coprime to~$\a\n$.  Then
$x\in\b_2^{-1}$, and since $\OK=x\b_2+\a\n=x\b_2+z\b_1^{-1}$, there
exist $w\in \b_2$ and~$y\in\b_1^{-1}$ such that $xw-yz=1$.  Now
$V=\mat{x}{y}{z}{w}\in\mat{\b_2^{-1}}{\b_1^{-1}}{\n\b_1}{\b_2}$ with
$\det V=1$ as required.

Finally, if~$P=(L,L')$ is a modular point with class~$c_{ij}$
and~$j>1$, we form $\q_j^{-1}P=(\q_j^{-1}L,\q_j^{-1}L')$, which has
class~$c_{i1}$, and simply observe that any admissible basis for the
latter is also an admissible basis for~$P$.
\end{proof}

Let
\[
  \Go^{\p}(\n) = \left\{\mat{a}{b}{c}{d}\mid a,d\in \OK;
  b\in\p^{-1}; c\in\n\p; ad-bc\in \OK^\times \right\}.
\]
It is easy to see that the group of matrices~$\gamma\in\GL(2,K)$ which
stabilize~$P_{ij}$ is~$\Go^{\p_i}(\n)$, from which the following is
immediate.

\begin{prop} \label{prop:ad-basis-unique-0}
Let $P$ be a modular point for $\Gon$ in class~$c_{ij}$, and
let $U$ be an admissible basis matrix for~$P$, so that $P=P_{ij}U$.
Then the set of all admissible basis matrices for~$P$ is the coset
$\Go^{\p_i}(\n)U$.\qed
\end{prop}

Again, the situation is simpler for square ideal classes, for which
the right stabilizer of the standard modular point is the same as for
principal modular points, namely~$\Gon$ itself.  It can be shown that
the more general groups~$\Go^{\p}(\n)$ defined above are conjugate to
$\Gon$ in $\GL(2,K)$ if and only if the ideal class~$[\p]$ is a
square.  More generally, if $[\p_1]=[\p_2][\a]^2$ then
$\Go^{\p_2}(\n)=M\Go^{\p_1}(\n)M^{-1}$ where~$M\in\GL(2,K)$ satisfies
$(\a\p_2\oplus\a)M=\p_1\oplus \OK$ and simultaneously
$(\a\p_2\oplus\a\n^{-1})M=\p_1\oplus \n^{-1}$.  Our choice of standard
lattices and modular points has been made to minimise the number of
different groups of the form~$\Go^{\p}(\n)$ which need to be
considered.

\begin{cor} \label{cor:g0n-bij}
For each ideal class~$c_{ij}=[\p_i\q_j^2]$,
\[
  \M_0^{(c_{ij})}(\n) = \{P_{ij}U \mid U\in\G\},
\]
and the correspondence $P_{ij}U \leftrightarrow U$ is a bijection
\[
   \M_0^{(c_{ij})}(\n) \longleftrightarrow \Go^{\p_i}(\n)\backslash\G.
\]
\end{cor}

\subsection{Modular points for \texorpdfstring{$\Gin$}{Gamma1(n)}}

A modular point for~$\Gin$ consists of a modular point for
$\Gon$ with some extra structure.  
\begin{defn}
A \emph{modular point for $\Gin$} is a triple $P=(L,L',\beta)$ where
$(L,L')$ is a $\Gon$-modular point (so $L$, $L'$ are lattices with
$L'\supseteq L$ and $L'/L\cong \OK/\n$), and $\beta\in L'$ generates
$L'/L$ as~$\OK$-module (that is, $L'=L+\OK\beta$).  We call $L$ the
\emph{underlying lattice} of~$P$, $[L]$ its \emph{class}, and $(L,L')$
the underlying $\Gon$-modular point.

We identify $(L,L',\beta_1)$ and~$(L,L',\beta_2)$ when
$\beta_1-\beta_2\in L$; in particular, $(L,L',\beta)=(L,L',u\beta)$
for $u\in\OK$ coprime to~$\n$.  The condition that~$\beta\pmod{L}$
generates~$L'/L$ is equivalent to $r\beta\in L\iff r\in\n$,
for~$r\in\OK$.

The set of modular points for~$\Gin$ is denoted~$\M_1(\n)$, and is the
disjoint union of the subsets~$\M_1^{(c)}(\n)$ of modular points in
each ideal class~$c$.
\end{defn}

If $P=(L,L',\beta)\in\M_1(\n)$, then also $PU=(LU,L'U,\beta
U)\in\M_1(\n)$ for all~$U\in\G$.  Further transformations on
modular points are defined below.

Together with each of the standard lattices~$L_{ij}$, we have already
associated a standard $\Gon$-modular point $P_{ij}$.  We now extend
each of these to a standard $\Gin$-modular point.  As
in~\Cref{sec:basic}, we fix $n_0\in\n^{-1}$ such that
$\n^{-1}=\<1,n_0>$, or equivalently such that $n_0\pmod{\OK}$
generates $\n^{-1}/\OK$ as $\OK$-module, and set $\beta_0=(0,n_0)\in
\OK\oplus\n^{-1}$.  Then the standard modular points are as follows.
For each~$j$, let~$\a_j$ be an ideal coprime to~$\n$ in the inverse
class to~$\q_j$, so that $\a_j\q_j=\<z_j>$ with~$z_j\in \OK$.  Then
multiplication by~$z_j$ induces an $\OK$-module isomorphism
$\n^{-1}/\OK\to\q_j\n^{-1}/\q_j$.  Set $n_j=n_0z_j\in\q_j\n^{-1}$ and
$\beta_j=(0,n_j)$; then $n_j$ generates the cyclic module
$\q_j\n^{-1}/\q_j$.  We define
\[
   \widetilde{P}_{ij} = (L_{ij},L_{ij}',\beta_j) = 
  (\q_j(\p_i\oplus \OK),\q_j(\p_i\oplus\n^{-1}),\beta_j). 
\]
Although this does depend on various choices made, we regard these as
fixed once and for all.  Also, since~$\beta_j$ depends only on~$j$ and
not on~$i$, it only depends on the class modulo squares.  In
particular, when the class group has exponent~$2$ we have
$\beta=\beta_0$ in all cases\footnote{This is one place where the
situation is not simpler for fields of odd class number.}.

Any other generator of $L_{ij}'/L_{ij}$ has the form
$\beta=u\beta_j\pmod{L_{ij}}$ where $u\in \OK$ is coprime to~$\n$.  If we let
$\gamma_u\in\Go^{\p_i}(\n)$ be any matrix with~$(2,2)$-entry~$u$,
then
\[
     (L_{ij},L_{ij}',\beta) = (L_{ij},L_{ij}',\beta_j)\gamma_u =
     \widetilde{P}_{ij}\gamma_u
\]
since $\gamma_u$ stabilises $(L_{ij},L_{ij}')$
by~\Cref{prop:ad-basis-unique-0}. Moreover,
\[
   \widetilde{P}_{ij}\gamma_u = \widetilde{P}_{ij} \iff 
   u\equiv1\pmod{\n} \iff
   \gamma_u\in\Gi^{\p_i}(\n),
\]
where
\[
   \Gi^{\p}(\n) = \left\{ 
        \mat{a}{b}{c}{d}\in\Go^{\p}(\n)\mid d-1\in\n
                       \right\}. 
\]
Hence the set of all extensions of~$P_{ij}$ to a $\Gin$-modular point
is the set of~$P_{ij}\gamma$, as $\gamma$ runs through
$\Gi^{\p_i}(\n)\backslash\Go^{\p_i}(\n)$.  This has
cardinality~$\varphi(\n)$: the index
$[\Go^\p(\n):\Gi^\p(\n)]=\varphi(\n)$ always, since $\Gi^\p(\n)$ is
the kernel of the surjective homomorphism
$\epsilon:\Go^\p(\n)\to(\OK/\n)^\times$ which maps a matrix to
its~$(2,2)$-entry.

More generally, to describe the set of all $\Gin$ modular
points with an arbitrary underlying $\Gon$ modular point we
first extend the notion of admissible basis.

\begin{defn}
An \emph{admissible basis matrix} for the modular point~$\widetilde{P}
= (L,L',\beta)\in\M_1^{(c_{ij})}(\n)$ is a matrix $U\in\G$ such that
$\widetilde{P} = \widetilde{P}_{ij}U$; that is, such that
\[
   L=\q_j(\p_i\oplus \OK)U; \qquad L'=\q_j(\p_i\oplus \n^{-1})U;  \qquad
   \beta=\beta_jU.
\]
The rows of an admissible basis matrix are an \emph{admissible basis}
for the modular point.
\end{defn}

\begin{prop} \label{prop:ad-basis-exist-1}
Every modular point~$\widetilde{P}$ for~$\Gin$ has an admissible
basis matrix~$U$; the set of all admissible basis matrices for~$\tilde
P$ is the coset $\Gi^{\p_i}(\n)U$ where~$c_{ij}$ is the class
of~$\widetilde{P}$; the set of $\Gin$-modular points with the same
underlying $\Gon$-modular point as~$\widetilde{P}$ is $\{\tilde
PU^{-1}\gamma U\mid \gamma\in\Go^{\p_i}(\n)\}$.

Hence
\[
  \M_1^{(c_{ij})}(\n) = \{\widetilde{P}_{ij}U \mid U\in\G\},
\]
and the correspondence $\widetilde{P}_{ij}U \leftrightarrow U$ is a bijection
\[
   \M_1^{(c_{ij})}(\n) \longleftrightarrow \Gi^{\p_i}(\n)\backslash\G.
\]
\end{prop}
\begin{proof}
Let $\widetilde{P}$ be a $\Gin$-modular point with underlying
$\Gon$-modular point~$P$.  Let $U$ be an admissible basis
matrix for~$P$.  Then~$\widetilde{P}U^{-1}=(\q_j(\p_i\oplus
\OK),\q_j(\p_i\oplus \n^{-1}),\beta)$ for some~$\beta$.  As seen above,
there exists $\gamma\in\Go^{\p_i}(\n)$ with
$\beta=\beta_j\gamma$.  Then $\widetilde{P}=\widetilde{P}_{ij}\gamma U$, so
$\gamma U$ is an admissible basis matrix for~$\widetilde{P}$.  The other
parts are clear.
\end{proof}

\section{The Hecke algebra} \label{sec:Hecke}

Let $\M=\M_0(\n)$ or~$\M_1(\n)$ be the set of modular points for
either~$\Gon$ or~$\Gin$.  For any field~$F$ we denote by~$F\M$ the
$F$-vector space generated by~$\M$; this is the direct sum of
subspaces~$F\M^{(c)}$, as $c$ runs over the ideal classes.  Elements
of~$F\M$ are formal finite linear combinations of modular points.

In this section we define an algebra~$\T$ of commuting linear
operators called Hecke operators, which act on~$\Q\M$, and prove their
formal properties.  It is possible to define a smaller algebra
acting on~$F\M$, where~$F$ is a field of finite characteristic~$p$,
but we do not do this here.  We also define Atkin-Lehner operators
on~$\Q\M_0(\n)$, and operators $A_{\d}:\M_0(\n)\to\M_0(\m)$ between
modular points of different levels, where $\m\mid\n$ and
$\d\mid\m^{-1}\n$.  Later we will consider the
complexification~$\C\M$, for example when considering eigenvectors for
the Hecke action, but we start with rational scalars, since Hecke
operators respect the rational structure.

\subsection{Formal Hecke operators}

We will only sketch the proofs of the main results here, since they
are almost identical to those for the classical theory; full details
were given in Bygott's thesis~\cite{JBthesis}, though our
notation differs slightly from there.

First we make a general observation.  If $(L,L')$ is a $\Gon$-modular
point and $M\supseteq L$ is a lattice with index~$[M:L]=\a$, we can
attempt to construct a new modular point $(M,M')$ with underlying
lattice~$M$ by setting $M'=M+L'$.  This is valid when $\a$ is coprime
to~$\n$, since then $L'\cap M=L$, so that $M'/M\cong L'/L'\cap
M=L'/L\cong \OK/\n$, but not in general.  In this situation, when we
talk of the modular point~$(M,M')$, we will always mean the one with
$M'=M+L'$, assuming that this is a valid modular point.  Similarly for
$\Gin$-modular points.

Each of the following operators on~$\Q\M_1(\n)$ will be defined by
specifying the image of each modular point in~$\M_1(\n)$ as a
$\Q$-linear combination of modular points, and extending by
$\Q$-linearity:
\begin{itemize}
\item an operator~$T_{\a,\a}$ for each fractional ideal~$\a$ coprime
  to~$\n$;
\item an operator~$T_{\a}$ for each integral ideal~$\a$;
\item an operator~$[\alpha]$ for each~$\alpha\in K^\times$ coprime to~$\n$.
\end{itemize}
In each case, we also obtain an operator on~$\Q\M_0(\n)$ (which is
trivial in the case of the~$[\alpha]$), by ignoring the third component
of~$\M_1(\n)$-modular points.  Hence, in these definitions, we may
restrict our attention to~$\M=\M_1(\n)$.

We will also define, only on~$\Q\M_0(\n)$,
\begin{itemize}
  \item an Atkin-Lehner operator $W_{\q}$ for each $\q\mid\mid\n$
\end{itemize}
as the linear extension of a map $\M_0(\n)\to\M_0(\n)$.

In~\Cref{sec:level-changing}, we will define operators which change
the level.  For all integral ideals~$\m,\n,\d$ with $\m\mid\n$ and
$\d\mid\m^{-1}\n$, we define an operator $A_{\d}:\M_0(\n)\to\M_0(\m)$.

When it is necessary to specify the level, we will
write~$T_{\a,\a}^{\n}$, $T_{\a}^{\n}$, and $W_{\q}^{\n}$ for the
operators at level~$\n$.

\begin{defn}[{The operators $T_{\a}$}]
  \label{def:Ta}
For each integral ideal $\a$, the operator $T_\a$ is defined by
\[
  T_\a(L,L',\beta) = {N(\a)}^{-1}\sum_{\substack{[M:L]=\a\\ (M,M',\beta)\in\M}}(M,M',\beta).
\]
\end{defn}
Here the sum is over all~$(M,M',\beta)$ such that $M$ is a
super-lattice of~$L$ with $[M:L]=\a$ and $M'=M+L'$ is such that
$M'/M\cong\OK/\n$, in which case it is automatic that $\beta$
generates~$M'/M$.  This sum is certainly finite, since $L\subseteq
M\subseteq\a^{-1}L$ and $\a^{-1}L/L$ is finite, so has only finitely
many submodules.

Clearly, $T_{\OK}$ is the identity operator.

\begin{defn}[{The operators $T_{\a,\a}$}]
  \label{def:Taa}
Here, the definition for~$\M_0(\n)$ is simpler.  We define
$T_{\a,\a}:\Q\M_0(\n)\to\Q\M_0(\n)$ for all fractional ideals~$\n$, by
setting
\[
T_{\a,\a}(L,L') = {N(\a)}^{-2}(\a^{-1}L,\a^{-1}L').
\]

To extend this definition to~$\Q\M_1(\n)$, it is necessary to restrict
to fractional ideals~$\a$ which are coprime to~$\n$.  Note that for
each~$(L,L',\beta)\in\M_1$, we may assume that $\beta\in
L'\cap\a^{-1}L'$, on replacing~$\beta$ by~$u\beta$ where $u\in\OK$
satisfies $u\equiv1\pmod\n$ and $u\in\a^{-1}$; such a~$u$ exists since
$\OK\cap\a^{-1}$ and~$\n$ are coprime.

\begin{lem}
Let $(L,L',\beta)\in\M_1(\n)$.  If $\beta\in\a^{-1}L'$, then also
$(\a^{-1}L,\a^{-1}L',\beta)\in\M_1(\n)$.
\end{lem}
\begin{proof}
To check that $(\a^{-1}L,\a^{-1}L',\beta)\in\M_1(\n)$, we must show
that $\a^{-1}L' = \OK\beta + \a^{-1}L$, so that $\beta$ generates
$\a^{-1}L'/\a^{-1}L$.

We first claim that $L'=(\OK\cap\a)L'+L$.  Since $\n L'\subseteq
L\subseteq L'$, and $\a$ and~$\n$ are coprime, we have
\[
L' = (\OK\cap\a)L'+\n L' \subseteq (\OK\cap\a)L'+L \subseteq L'+L=L'.
\]
Next, given that $L'=\OK\beta+L$ and $\a\beta\subseteq L'$, we have
$L'=(\OK+\a)\beta+L$, and hence
\[
L' = (\OK\cap\a)(\OK+\a)\beta + ((\OK\cap\a)+\OK)L = \a\beta+L.
\]
This implies that $\a^{-1}L' = \OK\beta + \a^{-1}L$, as required.
\end{proof}

Assuming the condition~$\beta\in L'\cap\a^{-1}L'$, setting
\[
T_{\a,\a}(L,L',\beta) = {N(\a)}^{-2}(\a^{-1}L,\a^{-1}L',\beta)
\]
gives a well-defined operator~$T_{\a,\a}:\Q\M_1(\n)\to\Q\M_1(\n)$ for
all fractional ideals coprime to~$\n$, {\it i.e.}, for
$\a\in\J_K^{\n}$.  From the definition, it is clear that $T_{\OK,\OK}$
is the identity, and that for~$\a,\b\in\J_K^{\n}$ we have
\[
T_{\a\b,\a\b} =  T_{\a,\a} T_{\b,\b}.
\]
\end{defn}

\begin{defn}[{The operators $[\alpha]$}]
  \label{def:diamond}
  Let $\alpha\in \OK$ be coprime to~$\n$.  The operator~$[\alpha]$ is defined by
  \[
    [\alpha](L,L',\beta) = (L,L',\alpha\beta).
    \]
    Hence $[\alpha]$ permutes the~$\Gin$-modular points
    above each~$\Gon$-modular point, and acts trivially on the latter.
\end{defn}
This definition is valid, since multiplication by~$\alpha$ induces an
automorphism of~$L'/L$.  Since $\alpha\beta\pmod{L}$ only depends
on~$\alpha\pmod{\n}$ we may regard~$[\alpha]$ as defined for
$\alpha\in(\OK/\n)^\times$, and hence extend it to all~$\alpha\in
K^\times$ coprime to~$\n$.

An alternative for the definition is to set $[\alpha](L,L',\beta) =
(L,L',\beta)\gamma_{\alpha}$ where $\gamma_d\in\Gon$ satisfies
$\epsilon(\gamma_{\alpha})=\alpha$ (recall that~$\epsilon$ maps a matrix to
its~$(2,2)$-entry).

\begin{defn}[{Atkin-Lehner operators $W_{\q}$}]
  \label{def:AL}

Let $\n=\q\q'$ be a factorization of the level~$\n$ with $\q,\q'$
coprime, so~$\q\mid\mid\n$, and let~$(L,L')\in\M_0(\n)$.

Let $L_1=L+\q'L'$ and~$L_2=L+\q L'$, so that $(L,L')=(L_1\cap
L_2,L_1+L_2)$.  Here, $L_1$ has index~$\q'$ in~$L'$ and contains~$L$
with index~$\q$, while for~$L_2$ the roles of~$\q,\q'$ are reversed:
$L_1/L\cong L'/L_2\cong\OK/\q$ and~$L_2/L\cong L'/L_1\cong\OK/\q'$.
Setting~$L_1'=\q^{-1}L_2$, we have~$(L_1,L_1')\in\M_0(\n)$.  Now we
define $W_{\q}:\Q\M_0(\n)\to\Q\M_0(\n)$ by setting
\[
W_{\q}(L,L') = N(\q)^{-1}(L_1,L_1') = N(\q)^{-1}(L+\q'L', \q^{-1}L+L').
\]

With respect to bases, if $L=(\a\oplus\b)U$
and~$L'=(\a\oplus\n^{-1}\b)U$, then $L_1=(\a\oplus\q^{-1}\b)U$,
$L_2=(\a\oplus\q'^{-1}\b)U$, so
$L_1'=(\q^{-1}\a\oplus\n^{-1}\b)U$. Ignoring the scaling factor
of~$N(\q)$ and the basis matrix~$U$ we have
\[
W_{\q} \colon
(\a\oplus\b,\a\oplus\n^{-1}\b)
\mapsto
(\a\oplus\q^{-1}\b, \q^{-1}\a\oplus\n^{-1}\b).
\]
As a special case, with $\q=\n$ we obtain the \emph{Fricke operator
$W_{\n}$}, which maps $(L,L')\mapsto N(\n)^{-1}(L',\n^{-1}L)$, and
satisfies~$W_{\n}^2=T_{\n,\n}$.
\end{defn}

\subsection{Relations between the operators}
\label{subsec:Hecke-relations}
\begin{prop} \label{prop:Hecke-relations}
The operators $T_\a$ and $T_{\a,\a}$ at a fixed level~$\n$ satisfy the
following identities:
\begin{enumerate}
\item $T_{\OK}=T_{\OK,\OK}=$ identity.
\item For all~$\a,\b$ coprime to~$\n$,
\[
  T_{\a,\a}T_{\b,\b} = T_{\a\b,\a\b} = T_{\b,\b}T_{\a,\a};
\]
in particular $T_{\a,\a}$ and~$T_{\b,\b}$ commute.
\item For all~$\a,\b$ with $\a$ coprime to~$\n$,
\[
  T_{\a,\a}T_{\b} = T_{\b}T_{\a,\a}.
\]
\item If $\a$ and $\b$ are coprime then
\[
   T_\a T_\b=T_{\a\b}=T_\b T_\a;
\]
in particular $T_\a$ and~$T_\b$ commute.
\item If $\p$ is a prime dividing~$\n$ then for all~$n\ge1$,
\[
  T_{\p^n} = (T_\p)^n.
\]
\item If $\p$ is a prime not dividing~$\n$ then for all~$n\ge1$,
\[
  T_{\p^n}T_\p = T_{\p^{n+1}} + N(\p)T_{\p^{n-1}}T_{\p,\p} .
\]
\end{enumerate}
\end{prop}
\begin{proof}
In all cases it suffices to compare the action on $\Gon$-modular
points since it is clear that the condition on~$\beta$ is satisfied.

(1) and (2) are clear.  For (3), both sides map~$(L,L')$
to~$N(\a^2\b)^{-1}$ times the sum of the modular
points~$(M,M')=(\a^{-1}M_1,\a^{-1}M_1')$ such that $[M:\a^{-1}L] = [\a
  M:L] = [M_1:L] = \b$.

For (4), first observe that each tower $L\subseteq M\subseteq N$ with
$[M:L]=\a$ and $[N:M]=\b$ such that $(M,M')=(M,L'+M)\in\M$ and
$(N,N')=(N,M'+N)\in\M$ gives $L\subseteq N$ with $[N:L]=\a\b$ and
$(N,L'+N)\in\M$.  Conversely, every extension $L\subseteq N$ with
$[N:L]=\a\b$ and $(N,N')=(N,L'+N)\in\M$ arises uniquely this way for a
suitable intermediate~$M$, which automatically satisfies
$(M,L'+M)\in\M$.

For (5) we show by induction that $T_{\p^{n+1}}=T_{\p}T_{\p^n}$ for
$n\ge1$.  The only part which is not quite obvious is that to each
extension $L\subseteq N$ with $[N:L]=\p^{n+1}$ with $(N+L')/N\cong
\OK/\n$, there is a unique intermediate~$M$ with $[M:L]=\p$ and
$(M+L')/L\cong \OK/\n$.  The fact that $(N+L')/N$ has exact
annihilator~$\n$ and $\p\mid\n$ implies that $L\not\subseteq\p N$, and
the unique such~$M$ is $M=L+\p N$.

For (6), each side when applied to $(L,L')\in\M$ is a linear
combination of $(M,M')$ where $[M:L]=\p^{n+1}$.  If the multiplicity
of~$(M,M')$ in $T_{\p^n}T_\p(L,L')$ is~$a$ and in
$T_{\p^{n-1}}T_{\p,\p}(L,L')$ is~$c$ then the identity follows from
$a=1+N(\p)c$ since the multiplicity in $T_{\p^{n+1}}(L,L')$ is~$1$ by
definition.  Either $M\supseteq\p^{-1}L$, in which case $c=1$ and
$a=N(\p)+1$; or $M\not\supseteq\p^{-1}L$, in which case $c=0$ and
$a=1$.
\end{proof}

\begin{defn}
  \label{def:T}
  For each level~$\n$, the algebra~$\T$ of formal Hecke operators on
  $\Q\M$ is the $\Q$-algebra generated by $T_{\a}$, for all integral
  ideals~$\a$, and $T_{\a,\a}$, for all fractional ideals~$\a$ coprime
  to~$\n$.

Note that the algebra~$\T$ depends on the level.
\end{defn}

\begin{cor}
The $\Q$-algebra $\T$ is commutative and generated (as a $\Q$-algebra)
by $T_\p$ (for all primes~$\p$) and $T_{\p,\p}$ (for all primes
$\p\nmid\n$).  As a module, $\T$ is generated by the operators of the
form~$T_{\a,\a}T_{\b}$.\qed
\end{cor}

The multiplicative relations between these operators may be summarised
in terms of a formal Dirichlet series:
\begin{cor}
We have the formal identity between the formal Hecke operators:
\[
   \sum_{\a}T_\a N(\a)^{-s} = \prod_{\p\mid\n}(1-T_\p N(\p)^{-s})^{-1}
   \prod_{\p\nmid\n}(1-T_\p N(\p)^{-s}+T_{\p,\p}N(\p)^{1-2s})^{-1}.
\]
Here the sum on the left is over all integral ideals~$\a$ of~$\OK$,
and the product on the right is over all prime ideals~$\p$.\qed
\end{cor}

Properties (1) and~(2) imply that the $T_{\a,\a}$ for~$\a$ coprime
to~$\n$ are invertible, and form a subgroup of the group of invertible
elements of~$\T$, namely the image of~$\J_K^{\n}$ under the group
homomorphism $\a\mapsto T_{\a,\a}$.

\subsubsection{Atkin-Lehner relations}
Firstly, as already noted,
\[
W_\q^2 = T_{\q,\q},
\]
since both map~$(L,L')$ to~$N(\q)^{-2}(\q^{-1}L, \q^{-1}L')$.  More
generally, suppose that for~$i=1,2$ we have exact
divisors~$\q_i\mid\mid\n$; then
\[
W_{\q_1}W_{\q_2} = T_{\q,\q}W_{\q_3},
\]
where $\q=\q_1+\q_2$ (the greatest common divisor of~$\q_1$
and~$\q_2$) and $\q_3=\q_1\q_2\q^{-2}$ .

\begin{prop} \label{prop:AL-Hecke-commute}
For all~$\a$ coprime to~$\n$ and all~$\q||\n$, the operators~$T_{\a}$
and~$W_\q$ on~$\M_0(\n)$ commute.
\end{prop}
\begin{proof}
Let~$\a$ be coprime to~$\n$ and~$\q\mid\mid\n$.  Set~$\q'=\n\q^{-1}$.

Let~$(L,L')$ be a modular point for $\Gon$.  Both~$T_{\a}W_{\q}(L,L')$
and~$W_{\q}T_{\a}(L,L')$ are equal to a scaling factor~$N(\a\q)^{-1}$
times a formal sum of modular points. For ease of notation we will
ignore the scaling factors in the rest of the proof, showing that the
two formal sums contain exactly the same modular points.

First applying~$W_{\q}$, the image of~$(L,L')$ is (in the notation
used in the definition)~$(L_1,L_1') = (L+\q'L',\q^{-1}L+L')$.  As
noted above in the definition of~$W_{\q}$, $L_1$ is the (unique)
sublattice of~$L'$ containing~$L$ with relative index~$\q$.

The image of~$(L,L')$ under~$T_{\a}$ is a sum of terms~$(M,M')$,
where~$M$ runs over all super-lattices of~$L$ with relative
index~$\a$, and~$M'=M+L'$.  We apply~$W_{\q}$ to each term.

Setting~$M_1 = M+\q'M' = M +\q'L'$ and~$M_2 = M+\q M' = M +\q L'$, the
image of~$(M,M')$ under~$W_{\q}$ is~$(M_1,\q^{-1}M_2)$. By
construction, this is again a modular point in~$\M_0(\n)$.  It remains
to show that the super-lattices of~$L_1$ of relative index~$\a$ are
precisely those of the form~$M_1 = M+\q'L'$ for some~$M$ with~$[M:L]=\a$.

Since~$\a$ and~$\q$ are coprime and~$[L_1:L]=\q$, every lattice
containing~$L_1$ with relative index~$\a$ is a super-lattice of~$L$ of
relative index~$\a\q$ and has the form~$L_1+M = (L+\q'L')+M = M+\q'L'$
where~$[M:L]=\a$, completing the proof.
\end{proof}

\subsection{Level-changing operators}
\label{sec:level-changing}

Here we only consider the spaces~$\M_0(\n)$.

\begin{defn}[Level-changing operators $A_{\d}$]
  \label{def:A_d}
  Let $\m,\n,\d$ be integral ideals with $\m\mid\n$ and
  $\d\mid\m^{-1}\n$.  The operator $A_{\d}:\M_0(\n)\to\M_0(\m)$ is
  defined by
  \[
  A_{\d}(L,L') = N(\d)(L\cap\d L', (L\cap\d L')+\m^{-1}\n L').
  \]
  This definition needs some explanation.   To see that the right-hand
  side does lie in~$\M_0(\m)$:
  \[
  \frac{(L\cap\d L')+\m^{-1}\n L'}{L\cap\d L'} \cong \frac{\m^{-1}\n
    L'}{L\cap\m^{-1}\n L'} \cong \frac{L'}{\m\n^{-1}L\cap L'} \cong \OK/\m.
  \]
  Alternatively, if $(L,L') = (\a\oplus\b,\a\oplus\n^{-1}\b)$ then
  \[
  A_{\d}(L,L') = N(\d)(\d\a\oplus\b, \d\a\oplus\m^{-1}\b) \in \M_0(\m);
  \]
  hence, for a standard modular point we have
  \[
  A_{\d}(\q(\p\oplus\OK,\p\oplus\n^{-1})) = N(\d)(\q(\d\p\oplus\OK), \q(\d\p\oplus\m^{-1})).
  \]
  These two formulas show that~$A_{\d}$ has class~$[\d]$, in the sense
  that
  \[
  A_{\d}(\M_0^c(\n)) \subseteq \M_0^{c[\d]}(\m).
  \]
  In terms of basis matrices, $A_{\d}$ maps
  \[
  \q(\p\oplus\OK, \p\oplus\n^{-1})U \mapsto N(\d)(\q(\d\p\oplus\OK), \q(\d\p\oplus\n^{-1}))U.
  \]
  To confirm that this is well-defined, recall that the matrix~$U$ on
  the left-hand side is well-defined up to being multiplied on the
  left by an element of~$\Go^{\p}(\n)$, while the right-hand side
  is well-defined under left multiplication by an element
  of~$\Go^{\d\p}(\m)$, and we
  have~$\Go^{\p}(\n)\subseteq\Go^{\d\p}(\m)$.
\end{defn}
\begin{rmk}
If the class number is~$1$, such as in the classical case, writing the
principal ideals~$\n=(N)$, $\m=(M)$ and~$\d=(d)$, where~$M\mid N$
and~$d\mid N/M$, we recover the classical definition of the
``degeneracy operator'' $A_d$ whose definition depends on the
inclusion
\[
 \mat{d}{0}{0}{1}\Go(N) \mat{d}{0}{0}{1}^{-1} \subseteq  \Go(M).
\]
\end{rmk}

We now consider how the level-changing operators $A_{\d}:\M_0(\n)\to
M_0(\m)$ interact with the Hecke operators~$T_{\a}$ and $T_{\a,\a}$
(for~$\a$ coprime to~$\n$) and $W_{\q}$, at levels~$\n$ and~$\m$.
\begin{prop}
  \label{prop:Ad-commutes-Tp}
  Let $\m\mid\n$, $\d\mid\m^{-1}\n$ and let~$\p$ be a prime not
  dividing~$\n$.  Then $A_{\d}T_{\p}^{\n} = T_{\p}^{\m}A_{\d}$, as maps from $\Q
  M_0(\n)$ to~$\Q M_0(\m)$.
\end{prop}
\begin{proof}
  Let $(L,L')\in\M_0(\n)$ and let $(\tilde{L},\tilde{L'}) =
  A_{\d}((L,L')) = (L\cap\d L', (L\cap\d L') + \m^{-1}\n L')
  \in\M_0(\m)$.  It suffices to show that for all $M\supset L$ with
  $M/L\cong\OK/\p$ and~$M'=L+M$, so that $(M,M')$ is one of the terms
  in the sum defining~$T_{\p}^{\n}(L,L')$, the image
  $(\tilde{M},\tilde{M'})=A_{\d}((M,M'))$ is one of terms appearing in
  the sum defining~$T_{\p}^{\m}(\tilde{L},\tilde{L'})$.

  First consider the underlying lattices: we require that
  $\tilde{M}\supset\tilde{L}$, with index~$[\tilde{M}:\tilde{L}]=\p$.
  On the one hand we have $M\supset L\supseteq\tilde{L}$ with
  indices~$[M:L]=\p$ and~$[L:\tilde{L}]=\d$, and on the other
  hand~$M\supseteq\tilde{M}\supseteq\tilde{L}$ with the first
  index~$[M:\tilde{M}]=\d$, so the second index must be~$\p$.  (Here,
  the inclusion~$\tilde{L}\subseteq\tilde{M}$ follows from
  $\tilde{L}=L\cap\d L'\subseteq M\cap\d(L'+M) = M\cap\d M' =
  \tilde{M}$.)  Coprimality of $\d$ and~$\p$ also now gives
  $M=L+\tilde{M}$, so the map $M\mapsto\tilde{M}$ is injective.

  Secondly, $\tilde{L'} = \tilde{L}+\m^{-1}\n L' \subseteq
  \tilde{M}+\m^{-1}\n M' = \tilde{M'}$, and a similar comparison of
  indices gives $[\tilde{M'}:\tilde{L'}] = \p$; since $\d$ and~$\p$
  are coprime, we have $\tilde{M}+\tilde{L'}=\tilde{M'}$, so
  $(\tilde{M},\tilde{M'})$ is indeed one of the terms in
  $T_{\p}^{\m}(\tilde{L},\tilde{L'})$.
\end{proof}

An alternative proof is possible here, using suitable bases, though it
is necessary to use a different basis for each superlattice~$M$
containing~$L$ with index~$\p$.  Fixing such an~$M$, after a change of
basis we may assume (using the coprimality of~$\p$ and~$\n$) that
$(L,L') = (\a\oplus\b, \a\oplus\n^{-1}\b)$ and also $(M,M') =
(\a\oplus\p^{-1}\b, \a\oplus\p^{-1}\n^{-1}\b)$.  Now, omitting the
scaling factor~$N(\d)$, $A_{\d}$ maps~$(L,L')$ to $(\d\a\oplus\b,
\d\a\oplus\m^{-1}\b)$ and~$(M,M')$ to $(\d\a\oplus\p^{-1}\b,
\d\a\oplus\p^{-1}\m^{-1}\b)$, implying the result.

We use this alternative approach in proving the next two propositions,
which are simpler, as both sides of the equations to be checked
consist of a single element of~$\M_0(\m)$ (multiplied by a certain
scaling factor), instead of a linear combination of several elements.

\begin{prop}
  \label{prop:Ad-commutes-Taa}
  Let $\m\mid\n$, $\d\mid\m^{-1}\n$ and let~$\q$ be a fractional
  ideal.  Then $A_{\d}T_{\q,\q}^{\n} = T_{\q,\q}^{\m}A_{\d}$.
\end{prop}
\begin{proof}
  It suffices, after a choice of basis, to consider the image under
  both composite operators of~$(L,L')=(\a\oplus\b,\a\oplus\n^{-1}\b)
  \in \M_0(\n)$.  On both sides, the scaling factor
  is~$N(\d)N(\q)^{-2}$.  Omitting this, both composite operators
  map~$(L,L')$ to the same element of~$\M_0(\m)$:
    \begin{align*}
    (\a\oplus\b,\a\oplus\n^{-1}\b)
    &\xrightarrow{T_{\q,\q}^{\n}}
    (\q^{-1}\a\oplus\q^{-1}\b, \q^{-1}\a\oplus\q^{-1}\n^{-1}\b)\\
    &\xrightarrow{A_{\d}}
    (\q^{-1}\d\a\oplus\q^{-1}\b, \q^{-1}\d\a\oplus\q^{-1}\n^{-1}\b)
    \end{align*}
    and
    \begin{align*}
    (\a\oplus\b,\a\oplus\n^{-1}\b)
    &\xrightarrow{A_{\d}}
    (\d\a\oplus\b, \d\a\oplus\n^{-1}\b)\\
    &\xrightarrow{T_{\q,\q}^{\m}}
    (\q^{-1}\d\a\oplus\q^{-1}\b, \q^{-1}\d\a\oplus\q^{-1}\n^{-1}\b).
    \end{align*}
\end{proof}

Using parts~(4) and~(6) of~\Cref{prop:Hecke-relations}, we can now
extend~\Cref{prop:Ad-commutes-Tp} to~$T_{\a}$ for all integral~$\a$
coprime to~$\n$:
\begin{cor}
\label{cor:Ad-commutes-Ta}
Let $\m\mid\n$, $\d\mid\m^{-1}\n$ and let~$\a$ be an integral ideal
coprime to~$\n$.  Then $A_{\d}T_{\a}^{\n} = T_{\a}^{\m}A_{\d}$.
\end{cor}

Next we turn to the Atkin-Lehner operators.  Here, it is convenient,
and clearly sufficient, to restrict to the case where $\m^{-1}\n$ is a
prime power.  The following result generalises Lemma~2.7.1
of~\cite{JCbook2}.

\begin{prop}
  \label{prop:Ad-commutes-Wq}
  Let~$\n=\p^{\alpha+\beta}\m$, where~$\p$ is prime, so that, with
  $\d=\p^{\alpha}$, we have $A_{\d}:\M_0(\n)\to\M_0(\m)$.
  Set~$\d'=\p^{\beta}$, so~$\n=\d\d'\m$, and let~$q\mid\mid\m$.
  \begin{enumerate}
  \item If~$\p\nmid\q$ then $\q\mid\mid\n$ also, and
    \[
    A_{\d}W_{\q}^{\n} = W_{\q}^{\m}A_{\d}.
    \]
  \item If~$\p\mid\q$ then $\q' = \p^{\alpha+\beta}\q = \d\d'\q
    \mid\mid \n$, and
    \[
    A_{\d}W_{\q'}^{\n} = W_{\q}^{\m}T_{\d',\d'}^{\m}A_{\d'}.
    \]
    Note that the operator on the left is~$A_{\d}$, while that on the
    right is the `complementary' operator~$A_{\d'}$.
  \end{enumerate}
\end{prop}
\begin{proof}
  As before, it suffices to consider the action of both sides
  on~$(\a\oplus\b,\a\oplus\n^{-1}\b)$.
  \begin{enumerate}
  \item A straightforward computation shows that, omitting the common
    scaling factor~$N(\d)N(\q)^{-1}$,
    \begin{align*}
    (\a\oplus\b,\a\oplus\n^{-1}\b)
    &\xrightarrow{W_{\q}^{\n}}
    (\a\oplus\q^{-1}\b, \q^{-1}\a\oplus\n^{-1}\b)\\
    &\xrightarrow{A_{\d}}
    (\d\a\oplus\q^{-1}\b, \q^{-1}\d\a\oplus\m^{-1}\b)
    \end{align*}
    and also
    \begin{align*}
    (\a\oplus\b,\a\oplus\n^{-1}\b)
    &\xrightarrow{A_{\d}}
    (\d\a\oplus\b, \d\a\oplus\m^{-1}\b)\\
    &\xrightarrow{W_{\q}^{\m}}
    (\d\a\oplus\q^{-1}\b, \q^{-1}\d\a\oplus\m^{-1}\b).
    \end{align*}
    \item Since~$\q'=\d\d'\q$, both sides have the same scaling
      factor~$N(\d)N(\q')^{-1} = N(\d'\q)^{-1}$.  Omitting these, we
      have
    \begin{align*}
    (\a\oplus\b,\a\oplus\n^{-1}\b)
    &\xrightarrow{W_{\q'}^{\n}}
    (\a\oplus\q'^{-1}\b, \q'^{-1}\a\oplus\n^{-1}\b)\\
    &\xrightarrow{A_{\d}}
    (\a\oplus\d'^{-1}\q^{-1}\b, \q^{-1}\a\oplus\d'^{-1}\m^{-1}\b),
    \end{align*}
    while also
    \begin{align*}
    (\a\oplus\b,\a\oplus\n^{-1}\b)
    &\xrightarrow{A_{\d'}}
    (\d'\a\oplus\b, \d'\a\oplus\m^{-1}\b)\\
    &\xrightarrow{T_{\d',\d'}^{\m}}
    (\a\oplus\d'^{-1}\b, \a\oplus\d'^{-1}\m^{-1}\b)\\
    &\xrightarrow{W_{\q}^{\m}}
    (\a\oplus\d'^{-1}\q^{-1}\b, \q^{-1}\a\oplus\d'^{-1}\m^{-1}\b).
    \end{align*}

  \end{enumerate}
\end{proof}

\goodbreak
\subsection{Alternative normalisations and dual operators}
\label{subsec:alt-norm}
\subsubsection{Scaling}  It is possible to define the
operators~$T_{\a}$ and~$T_{\a,\a}$ without the scaling
factors~$N(\a)^{-1}$ and~$N(\a)^{-2}$ respectively.  The effect on the
formal Dirichlet series $\sum T_\a N(\a)^{-s}$ would be to shift the
variable~$s$ by~$1$.  This is the normalization used by Bygott
in~\cite{JBthesis}.  Similarly for $W_{\q}$.

\subsubsection{Using sublattices: dual operators}
In our definition of Hecke operators (see Definitions~\ref{def:Ta}
and~\ref{def:Taa} above), we map each lattice to one or more
superlattices; this follows the treatment (for $K=\Q$) by Lang
\cite{LangModularForms} and Koblitz \cite{KoblitzECMF}.  One can
alternatively use sublattices instead, as in Serre \cite{SerreCA} (for
$K=\Q$ and level~$1$).  The latter approach is slightly more
convenient for computation, as the operators can be described
explicitly in terms of $2\time2$ matrices with entries in~$\OK$, as we
show later in Section~\ref{sec:hecke-matrices}.  We summarise this
second approach here, defining new `dual' operators:
\begin{itemize}
  \item $\widetilde{T}_{\a}$ (for integral ideals~$\a$ coprime to~$\n$);
  \item $\widetilde T_{\a,\a}$ (for fractional ideals~$\a$ coprime to~$\n$);
  \item $\widetilde{W}_\q(L,L')$ (for~$\q\mid\mid\n$).
\end{itemize}

The operator~$\widetilde{T}_{\a}$ is defined for ideals~$\a$ coprime
to~$\n$ as follows.  Observe that whenever $L\subseteq M$ with
$[M:L]=\a$, then $L\subseteq M\subseteq\a^{-1}L$ with also
$[\a^{-1}L:M]=\a$, and hence $\a L\subseteq\a M\subseteq L$
where~$[L:\a M]=\a$.  Thus to each superlattice $M\supseteq L$ with
index~$\a$ there corresponds a sublattice~$\a M\subseteq L$ also with
index~$\a$.  This correspondence extends to modular points in either
$\M_0(\n)$ or~$\M_1(\n)$.

Now we set
\[
  \widetilde{T}_\a(L,L',\beta) = {N(\a)}\sum_{\substack{M\subseteq
      L\\ [L:M]=\a\\ (M,M',\beta')\in\M}}(M,M',\beta').
\]
Here, $M'=M+\a L'$, and the term~$(M,M',\beta')$ is only included when
$M'/M\cong \OK/\n$; also, $\beta'=u\beta\in\a L'\subseteq M'$ where
$u\in\a$ and~$u\equiv1\pmod{\n}$.  A simple calculation using the
correspondence between superlattices and sublattices shows that
\[
   \widetilde{T}_{\a} = T_{\a,\a}^{-1}T_{\a}.
\]

Next, for~$\a$ coprime to the level~$\n$, set
\[
\widetilde T_{\a,\a}=T_{\a^{-1},\a^{-1}},
\]
so that $\widetilde T_{\a,\a}((L,L',\beta))={N(\a)}^{2}(\a L,\a
L',\beta')$ with $\beta'$ defined as above.

These new operators satisfy the relations (1)--(4) and~(6)
of~\Cref{prop:Hecke-relations}.

Lastly, in the definition of Atkin-Lehner operators above
in~\Cref{def:AL}, we map the modular point~$(L,L')\in\M_0(\n)$
to~$(L+\q' L', \q^{-1}L+L')$, where the lattice $L_2=L+\q' L'$
satisfies $L\subseteq L_2\subseteq L'$. (Here, as above, $\n=\q\q'$
with $\q$, $\q'$ coprime.)  For the alternative normalization we
define
\[
\widetilde{W}_\q(L,L') = N(\q)(\q L+\n L', L + \q L'),
\]
noting that $\q L+\n L'\subseteq L$.  Then
\[
\widetilde{W}_\q = T_{\q,\q}^{-1}W_q
\]
and
\[
\widetilde{W}_{\q}^2 = \widetilde{T}_{\q,\q}.
\]

\begin{rmk}
  One reason for calling these ``dual operators'' is that, in a
  suitable context, they are the adjoint operators with respect to an
  inner product generalising the Petersson inner product: see
  section~\ref{subsec:level-changing}.

  For an alternative viewpoint on this duality, we can also define a
  ``Hecke correspondence'' on~$\M_1(\n)$, for each fixed ideal~$\a$,
  to be the subset~$T(\a)$ of~$\M_1(\n)\times\M_1(\n)$ consisting of
  pairs~$(P,Q)$ with $P=(L,L',\beta)$ and~$Q=(M,M',\gamma)$ such
  that~$L\subseteq M$ and $[M:L]=\a$.  Then the modular points
  appearing in the definition of~$T_{\a}(P)$ are precisely those~$Q$
  for which~$(P,Q)\in T(\a)$, while reciprocally, the modular points
  which appear in the definition of~$\widetilde{T}_{\a}(Q)$ are
  those~$P$ for which~$(P,Q)\in T(\a)$.
\end{rmk}

\subsection{Grading the Hecke algebra}

Each of the operators~$T$ we have just defined has a well-defined
\emph{ideal class}~$[T]$ such that $T(\M^{(c)}) \subseteq \M^{([T]c)}$ for
each~$c\in\Cl(K)$.  From the definitions, we see that
\begin{multicols}{2}
\begin{itemize}
\item $[T_{\a}] = [\a]$;
\item $[T_{\a,\a}] = [\a]^2$;
\item $[W_{\q}] = [\q]$;
\end{itemize}
\columnbreak
\begin{itemize}
\item $[\widetilde{T}_{\a}] = [\a]^{-1}$;
\item $[\widetilde{T}_{\a,\a}] = [\a]^{-2}$;
\item $[\widetilde{W}_{\q}] = [\q]^{-1}$.
\end{itemize}
\end{multicols}
In explicit computations it is simplest to work with \emph{principal}
operators, whose class is trivial.  The basic principal operators have
the form $T_{\a,\a}T_{\b}$ (or $T_{\a,\a}T_{\b}W_{\q}$) where $\a^2\b$
(or $\a^2\b\q$) is principal; they map each $\M^{(c)}$ to itself.  In
Section~\ref{sec:hecke-matrices} we will describe the action of
principal operators explicitly in terms of matrices in~$M_2(K)$, as is
done in the classical setting.

Note that $T_{\a,\a}T_{\b}$ (respectively, $T_{\a,\a}T_{\b}W_{\q}$) is
principal if and only if $\widetilde{T}_{\a,\a}\widetilde{T}_{\b}$
(respectively, $\widetilde{T}_{\a,\a}\widetilde{T}_{\b}\widetilde{W}_{\q}$) is.

Each element of~$\T$ is (non-uniquely) a linear combination of the
basic operators.  Since in each of the relations listed
in~\Cref{prop:Hecke-relations} all terms have the same class, it
follows that we have a well-defined decomposition
\begin{equation}
  \label{eqn:T-decomp}
  \T = \bigoplus_{c\in\Cl(K)}\T_c
\end{equation}
where~$\T_c$ is spanned by all basic operators with class~$c$.  By an
operator of class~$c$, we now mean any element of~$\T_c$; thus, every
Hecke operator is uniquely expressible as a sum of operators, one in
each graded component~$\T_c$, and operators~$T\in\T_c$ are
characterized by the property that
\begin{equation}
  \label{eqn:T-M-action}
   T(\Z\M^{(c')}) \subseteq \Z\M^{(cc')}\qquad\forall c'\in\Cl(K).
\end{equation}
This decomposition of the algebra~$\T$ is a grading by the abelian
group~$\Cl(K)$, on account of the following.
\begin{prop}
  \label{prop:T-grading}
  \[
  \T_{c_1} \T_{c_1} \subseteq \T_{c_1c_2}.
  \]
\end{prop}
\begin{proof}
This follows immediately from the preceding characterization
of~$\T_c$, or simply from the definition of the class of a basic
operator.
\end{proof}

\begin{cor}
The component $\T_1$, where $1$ denotes the trivial ideal class, is a
subalgebra of~$\T$.  More generally, if $H\le\Cl(K)$ is any subgroup
then $\T_H:=\oplus_{c\in H}\T_c$ is a subalgebra of~$\T$.\qed
\end{cor}

\section{Eigenvalue systems}\label{sec:eigenvalue-systems}
We continue to set $\M=\M_0(\n)$ or~$\M_1(\n)$, the set of modular
points for either~$\Gon$ or~$\Gin$.  Given a vector space on
which~$\T$ acts linearly, such as~$\C\M$, we may consider
simultaneous eigenvectors for all~$T\in\T$, and to such an eigenvector
we can associate a system of eigenvalues, one for each~$T\in\T$; this
determines an algebra homomorphism~$\T\to\C$.  In this section we
first study such eigenvalue systems or \emph{eigensystems} from a
purely formal point of view, and then apply the results obtained both
to~$\C\M$ and to \emph{formal modular forms}, which are vector-valued
functions on~$\M$.

Recall that a ray class character\footnote{Here we take~$\n$ to be an
ideal of~$\OK$ rather than a more general modulus with possible
components at real infinite places of~$K$, where $K$ is a number
field.  The theory we develop here could be extended to include
consideration of real places.} of~$K$ with modulus~$\n$ is a character
of the ray class group~$\J^\n_K/\PP^\n_K$, that is, a
homomorphism~$\chi:\J^\n_K\to\C^\times$, where~$\J^\n_K$ is the group
of fractional ideals of~$K$ coprime to~$\n$, which is trivial on the
subgroup~$\PP^\n_K$ of principal ideals generated by an
element~$a\equiv1\pmod{\n}$.  The restriction of~$\chi$ to principal
ideals gives a Dirichlet character
$(\OK/\n)^\times/\OK^\times\to\C^\times$, so that $\chi$ is an
extension of a Dirichlet character (which is trivial on units) by a
character of the ideal class group.  When~$\n$ is the unit ideal, one
obtains \emph{unramified characters} which are characters of the ideal
class group.

\begin{defn}
An~\emph{eigenvalue system of level~$\n$} is a ring homomorphism
\[
     \lambda: \T \to \C.
\]
An~\emph{eigenvalue system of level~$\n$ and character~$\chi$}, where
$\chi:\J^\n_K\to\C^\times$ is a character modulo~$\n$, is an
eigenvalue system having the additional property that
\[
     \lambda(T_{\a,\a})=\chi(\a)
\]
for all ideals~$\a$ coprime to~$\n$.
\end{defn}
Since $\T$ is generated (as a ring) by the operators $T_{\a}$ for
ideals $\a$ and $T_{\a,\a}$ for ideals~$\a$ coprime to~$\n$, to define
an eigenvalue system we must specify the values (in~$\C$ and~$\C^*$
respectively) of the two functions $\alpha:\a\mapsto\alpha(\a)=\lambda(T_\a)$ for all
ideals~$\a$, and $\chi:\a\mapsto\chi(\a)=\lambda(T_{\a,\a})$ for all ideals~$\a$
coprime to~$\n$.  These values must satisfy the following
multiplicative relations (numbered as in~\Cref{prop:Hecke-relations}):
\begin{enumerate}
\item $\alpha(\OK)=\chi(\OK)=1$.
\item For all~$\a,\b$,
\[
  \chi(\a\b) = \chi(\a)\chi(\b).
\]
\item Trivially, $\chi(\a)\alpha(\b)=\alpha(\b)\chi(\a)$.
\item If $\a$ and $\b$ are coprime then
\[
   \alpha(\a\b) = \alpha(\a)\alpha(\b).
\]
\item If $\p$ is a prime dividing~$\n$, then for all~$n\ge1$,
\[
  \alpha(\p^n) = \alpha(\p)^n.
  \]
\item If $\p$ is a prime not dividing~$\n$, then for all~$n\ge1$,
\[
  \alpha(\p^n)\alpha(\p) = \alpha(\p^{n+1}) +
  N(\p)\alpha(\p^{n-1})\chi(\p).
\]
\end{enumerate}
We write $\lambda=(\alpha,\chi)$ for the eigensystem determined by
functions~$\alpha$ and~$\chi$ satisfying these conditions.  We see
from~(2) that $\chi:\J_K^\n\to\C^*$ is a group homomorphism, and is
the character of the eigensystem~$\lambda$.  It is for this reason
that we defined $T_{\a,\a}$ for fractional ideals, since they form a
group rather than just a monoid, and so the values~$\chi(\a)$ are
forced to be nonzero.

\begin{rmk}
  If we were to set $\chi(\a)=0$ when $\a$ is not coprime to~$\n$,
  then these multiplicative relations could be expressed via a formal
  Dirichlet series and Euler product:
  \[
  \sum_\a \alpha(\a)N(\a)^{-s} =
  \prod_\p(1-\alpha(\p)N(\p)^{-s}+\chi(\p)N(\p)^{1-2s})^{-1}.
  \]
\end{rmk}

If $\lambda=(\alpha,\chi)$ is an eigenvalue system, and $\psi$ is a
character of~$\J_K^\n$, then we can define the \emph{twist}
$\lambda\otimes\psi$ of $\lambda$ by~$\psi$ to be the eigenvalue
system $(\alpha\psi,\chi\psi^2)$:
\begin{align*}
(\lambda\otimes\psi)(T_\a) &= \psi(\a)\lambda(T_\a) = \psi(\a)\alpha(\a); \\
(\lambda\otimes\psi)(T_{\a,\a}) &= \psi(\a)^2\lambda(T_{\a,\a}) = \psi(\a)^2\chi(\a).
\end{align*}
Hence, if~$\lambda$ has character~$\chi$ then $\lambda\otimes\psi$ has
character~$\chi\psi^2$.  In particular, the twist of~$\lambda$ by a
quadratic character~$\psi$ has the same character as~$\lambda$.  In
this situation, it is possible that $\lambda\otimes\psi=\lambda$, in
which case $\alpha(\a)=0$ for all~$\a$ such that $\psi(\a)=-1$. Such
an eigensystem is said to have \emph{inner twist} by $\psi$.

\begin{rmk}
It would be possible to restrict attention to eigensystems whose
characters have additional properties, such as being ray class
characters modulo~$\n$, or unramified characters, which may be viewed
as characters of the ideal class group.
\end{rmk}

\subsection{Restriction to principal operators}

In practical situations, where we are computing the action of Hecke
operators explicitly on some finite-dimensional vector space, it is
simpler to deal with principal operators, i.e. those in $\T_1$, as
these may be expressed using matrices in~$\Mat_2(K)$.  These matrices
will be described in detail in Section~\ref{sec:hecke-matrices}.  This
leads us to consider the question: to what extent is an eigenvalue
system~$\lambda$ determined by its restriction to~$\T_1$?  We will
give a complete theoretical answer to this question, and then make
some remarks about how to handle the situation in practice, if we only
compute principal Hecke operators but wish to determine the entire
eigenvalue system.

We continue to identify unramified characters of $\J_K^\n$ with
characters of the class group~$\Cl(K)$.

\begin{thm}
  \label{thm:eig-systems}
  Two eigenvalue systems $\lambda_1$ and $\lambda_2$ of the same level
  $\n$ have the same restriction to the subalgebra $\T_1$ of principal
  operators if and only if $\lambda_2=\lambda_1\otimes\psi$ for some
  unramified character~$\psi$.

  In this case, the characters~$\chi_1$, $\chi_2$ of $\lambda_1$, $\lambda_2$
  satisfy $\chi_2=\chi_1\psi^2$, and hence $\lambda_1$ and $\lambda_2$
  have the same character if any only if $\psi$ is quadratic.
\end{thm}

\begin{proof}
  One direction is clear: from the definition of twisting, we see that
  twisting by an unramified character~$\psi$ does not change the value
  of the eigensystem on principal ideals since $\psi$ is trivial on
  these.

  Assume that $\lambda_1=(\alpha_1,\chi_1)$ and
  $\lambda_2=(\alpha_2,\chi_2)$ have the same restriction to~$\T_1$.
  We first compare the characters~$\chi_1$, $\chi_2$. For ideals
  $\a\in\J_K^\n$ such that $\a^2$ is principal, the operator
  $T_{\a,\a}$ is principal, so
  \[
  \chi_1(\a) = \lambda_1(T_{\a,\a}) = \lambda_2(T_{\a,\a}) = \chi_2(\a).
  \]
  Hence $\chi_2\chi_1^{-1}$ is trivial on the $2$-torsion subgroup
  $\Cl[2]$, which implies that $\chi_2\chi_1^{-1}=\psi^2$ for some
  character $\psi$.  Replacing $\lambda_1$ by~$\lambda_1\otimes\psi$,
  we may now assume that $\chi_1=\chi_2$, and set $\chi=\chi_1$.

  Note that we can still twist $\lambda_1$ by an unramified quadratic
  character without changing~$\chi_1$; our aim now is to show that for
  some quadratic (or trivial) unramified character~$\psi$ we
  have~$\lambda_1\otimes\psi=\lambda_2$.  When the class number is
  odd, there are no quadratic unramified characters, so we must have
  that $\lambda_1=\lambda_2$ already in this case.

  For all ideals~$\a$ whose class is a square we have
  $\alpha_1(\a)=\alpha_2(\a)$; for if we take~$\b\in\J_K^\n$ such
  that~$\a\b^2$ is principal, then~$T_\a T_{\b,\b}\in\T_1$, so
  \[
  \alpha_1(\a)\chi(\b) = \lambda_1(T_\a T_{\b,\b}) = \lambda_2(T_\a
  T_{\b,\b}) = \alpha_2(\a)\chi(\b).
  \]
  Since $\chi(\b)\not=0$, this gives $\alpha_1(\a)=\alpha_2(\a)$, as
  claimed.  Hence the restrictions of~$\lambda_1$ and~$\lambda_2$
  to~$\T_{\Cl(K)^2}$ are equal.  This completes the proof in the case
  of odd class number, since then~$\Cl(K)^2=\Cl(K)$.

  The case of even class number requires extra work.

  Note that the restriction of every eigensystem~$\lambda =
  (\alpha,\chi)$ to~$\T_c$ cannot be identically zero for any
  class~$c$ which is a square, since these classes contain the
  operators~$T_{\a,\a}$, and $\lambda(T_{\a,\a})=\chi(\a)\not=0$.
  However, it is possible that for a non-square class~$c$ we may
  have~$\lambda(T)=0$ for all~$T\in\T_c$.  The set of classes~$c$ such
  that $\lambda$ is not identically zero on~$\T_c$ is a
  subgroup~$H_0(\lambda)$ of~$\Cl(K)$, containing~$\Cl(K)^2$, since if
  the restrictions of~$\lambda$ to two classes~$c_1$ and~$c_2$ are
  both not identically zero then (since $\lambda$ is fully
  multiplicative) its restriction to class~$\T_{c_1c_2}$ is also
  nonzero.

  We now compare the values of~$\alpha_1$ and~$\alpha_2$ on
  ideals~$\a$ in non-square classes.  We always have
  \[
  \alpha_1(\a)^2 = \lambda_1(T_\a)^2 = \lambda_1(T_\a^2) =
  \lambda_2(T_\a^2) = \alpha_2(\a)^2,
  \]
  since $T_\a^2$ has square class.  So $\alpha_2(\a)=\pm
  \alpha_1(\a)$.  This implies that $H_0(\lambda_1)=H_0(\lambda_2)$,
  and we denote this common subgroup by~$H_0$.  We claim that there is
  a quadratic character~$\psi$ of~$H_0$ such that
  $\alpha_2(\a)=\psi(\a)\alpha_1(\a)$ for all~$\a\in\J_K^\n$ whose
  classes lie in~$H_0$.  For $c\in H_0$, choose $\a\in c$
  with~$\alpha_1(\a) = \alpha_2(\a) \not= 0$, and define~$\psi(\a) =
  \alpha_2(\a)/\alpha_1(\a) = \pm1$.  We check that this sign is
  independent of the choice of~$\a\in c$, and gives a well-defined
  homomorphism $H_0\to\{\pm1\}$. To this end, we have:
  \begin{itemize}
  \item $\psi(\a)=1$ if $c\in \Cl(K)^2$, and in particular if~$\a$ is
    principal.
  \item $\psi(\a)=\psi(\a')$ for all~$\a,\a'\in c$
    with~$\alpha(\a),\alpha(\a')\not=0$, so $\psi$ is well-defined,
    only depending on the class~$c$. To see this, note that
    $T_{\a}T_{\a'}\in\T_{c^2}$, and hence
    \[
    \alpha_1(\a)\alpha_1(\a') = \lambda_1(T_{\a})\lambda_1(T_{\a'}) =
    \lambda_1(T_{\a}T_{\a'}) = \lambda_2(T_{\a}T_{\a'}) =
    \alpha_2(\a)\alpha_2(\a'),
    \]
    so
    \[
    \psi(\a')=\alpha_2(\a')/\alpha_1(\a')=\alpha_1(\a)/\alpha_2(\a)=\psi(\a)^{-1}=\psi(\a).
    \]
    (Here it is useful to remember that while the $\lambda_i$ are
    completely multiplicative as functions on Hecke operators, the
    $\alpha_i$ (as functions on ideals) are not, since in
    general~$T_{\a}T_{\a'}\not=T_{\a\a'}$; this is why we
    use~$T_{\a}T_{\a'}$ here rather than~$T_{\a\a'}$.)
  \item For $i=1,2$ let $\a_i\in c_i$ where $c_i\in H_0$ and
    $\alpha(\a_i)\not=0$.  Without loss of generality, we may assume
    that $\a_1,\a_2$ are coprime, so
    that~$\alpha_1(\a_1\a_2)=\alpha_1(\a_1)\alpha_1(\a_2)$ and
    similarly for~$\alpha_2$.  Then
    \[
    \psi(\a_1\a_2)
    = \frac{\alpha_2(\a_1\a_2)}{\alpha_1(\a_1\a_2)}
    = \frac{\alpha_2(\a_1)}{\alpha_1(\a_1)}\frac{\alpha_2(\a_2)}{\alpha_1(\a_2)}
    = \psi(\a_1)\psi(\a_2).
    \]
  \end{itemize}
Hence $\psi$ induces a well-defined quadratic character on~$H_0$,
trivial on~$\Cl(K)^2$, such that $\alpha_2(\a)=\psi(\a)\alpha_1(\a)$
for all~$\a\in\J_K^\n$ with class $[\a]\in H_0$.

This completes the proof in case $H_0=\Cl(K)$.

If $H_0$ is a proper subgroup of~$\Cl(K)$, we can extend the domain of
definition of~$\psi$ to all of~$\Cl(K)$ in $[\Cl(K):H_0]$ ways, giving
this number of unramified quadratic characters~$\psi$ extending the
character of $H_0$ defined above.  For each such extension we
have~$\lambda_2 = \lambda_1\otimes\psi$, since if $[\a]\notin H_0$
then~$\alpha_2(\a) = \alpha_1(\a) = 0$.

This completes the proof.  Note that in this last case (when
$H_0\not=\Cl(K)$) the character~$\psi$ is not unique, and $\lambda_1$
has inner twists by any quadratic character which is trivial on~$H_0$.
\end{proof}

In summary, this proof shows that when $\lambda_1$ and~$\lambda_2$
have the same restriction to the subalgebra~$\T_1$ of principal
operators, then we may first twist~$\lambda_1$ by an unramified
character so that its character matches that of~$\lambda_2$, and then
twist again by a quadratic unramified character to give~$\lambda_2$
precisely without affecting the character of the eigensystem further.
The overall twist required is unique except in case the $\lambda_i$
have inner twists; in that case, both $\lambda_i$ are identically zero
on operators in classes forming the complement of a proper
subgroup~$H_0$, and the number of inner twists (including the trivial
character) is the index~$[\Cl(K):H_0]$, which is a power of~$2$,
since~$\Cl(K)^2 \subseteq H_0 \subseteq \Cl(K)$.

\begin{rmk}
  \label{rmk:only-one-inner-twist}
  In this abstract setting, an eigensystem could have more than one
  inner twist; in other words, the index of~$H_0$ in~$\Cl(K)$ may be
  greater than~$2$.  However, in applications to automorphic forms
  there can only be one (non-trivial) inner twist,
  i.e.~$[\Cl(K):H_0]\le2$.  This follows from the existence of
  irreducible $2$-dimensional Galois representations associated to
  Hecke eigensystems, and the proof is beyond the scope of this paper.
\end{rmk}

\begin{rmk}
  \label{rmk:eigensystem-algorithm}
The proof of \Cref{thm:eig-systems} can be turned into an algorithm
for recovering a full eigensystem~$\lambda=(\alpha,\chi)$ with
unramified character from its restriction to~$\T_1$, unique up to
unramified twist.  Briefly, one first chooses an unramified
character~$\chi$ matching the values~$\lambda(T_{\a,\a})$ where $\a^2$
is principal; in the odd class number case this is no restriction,
and~$\chi$ can be arbitrary; in particular, we may take~$\chi$ to be
trivial. Then it suffices (in view of the multiplicative relations) to
define~$\alpha$ on prime ideals~$\p$.  When $\p$ has square class, use
the values~$\lambda(T_{\a,\a}T_{\p})$ and $\chi(\a)$, where $\a$ is
chosen so that $\a^2\p$ is principal, to define $\alpha(\p)$.
Otherwise, use the relation $\alpha(\p)^2=\alpha(\p^2)+N(\p)\chi(\p)$
to determine~$\alpha(\p)$ up to sign; here we may take an ideal~$\a$
such that $\a\p$ is principal to determine $\alpha(\p^2) =
\lambda(T_{\a,\a}T_{\p^2})/\chi(\a)$.  The signs must then be chosen
so as to preserve multiplicativity, which allows $r$ independent
choices where $[H_0:\Cl(K)^2]=2^r$.  We omit the details, an account
of which (in the Bianchi setting) may be found in~\cite{CTY}; code
implementing this procedure for Bianchi modular forms, where $K$ is an
arbitrary imaginary quadratic field, is included in the author's {\tt
  C++} package~\cite{bianchi-progs}.
\end{rmk}

We single out two special cases, being at opposite extremes.

\subsubsection{Special case: odd class number}   When $\Cl(K)$ has
odd order, the above situation simplifies somewhat, since all ideal
classes are squares, the class group has no quadratic characters, and
$H_0=\Cl(K)^2=\Cl(K)$.  Now, if two eigenvalue systems have the same
restriction to~$\T_1$ and the same character, then they are equal.
More generally, if $\lambda$ and~$\lambda'$ have the same restriction
to~$\T_1$ then $\lambda'=\lambda\otimes\psi$ for a unique unramified
character~$\psi$.  All eigensystems are unramified quadratic twists of
eigensystems with trivial character.

Concrete examples for this situation may be found in Lingham's 2005
thesis~\cite{LinghamThesis} for the fields~$\Q(\sqrt{-23})$
and~$\Q(\sqrt{-31})$, both of class number~$3$.  There, the interest
was in eigenvalue systems with trivial character, where~$T_{\a,\a}$
act trivially, and the eigenvalue systems are uniquely determined by
their principal eigenvalues.  To compute the eigenvalue of~$T_\p$ in
general in this case, one can choose~$\a$ coprime to~$\n$ such that
$\a^2\p$ is principal and compute the operator~$T_{\a,\a}T_{\p}$,
which is principal, and can be described using a finite set of
$N(\p)+1$ matrices in~$\Mat_2(K)$; then
$\lambda(\p)=\lambda(T_{\a,\a}T_{\p})$.  See the examples at the end
of~\Cref{sec:hecke-matrices}.  Many more examples for imaginary
quadratic fields with odd class numbers up to~$63$ have since been
computed by the author, and may be found in the LMFDB~\cite{lmfdb}.

\subsubsection{Special case: elementary abelian $2$-group}   In this case
$\Cl(K)^2$ is trivial, and every square class is principal.  Hence the
restriction of~$\lambda$ to~$\T_1$ completely determines its
character, and determines~$\lambda$ itself up to an unramified
quadratic twist.  Each restricted eigenvalue system extends to (in
general)~$h$ complete eigenvalue systems, forming a set of~$h$
unramified quadratic twists; in the presence of inner twists the
number is a proper divisor of~$h$, and the eigenvalues for all
operators in some classes are all zero.

Examples for this situation appear in Bygott's 1998
thesis~\cite{JBthesis} for~$\Q(\sqrt{-5})$, of class number~$2$.
There, eigenvalues systems with any unramified character were
considered, that is, both those with trivial character (called
``plusforms'' in~\cite{JBthesis}) and those with the unique unramified
quadratic character (called ``minusforms'').  In general the
eigenvalue systems come in pairs which are unramified quadratic twists
of each other and have the same character, though there are examples
of inner twist.  For example, at level~$\n=\<8>$ there are two
eigenvalue systems denoted~$f_{15}$ and~$f_{16}$ in~\cite{JBthesis},
each of which is its own quadratic twist, so has all eigenvalues for
non-principal Hecke operators equal to~$0$; of these, $f_{16}$ has
trivial character, while~$f_{15}$ has nontrivial character.

Although we have not yet considered the question of rationality of the
eigenvalues, it is interesting to note several cases
in~\cite{JBthesis} where the eigenvalues of principal operators are
rational integers while those of non-principal operators lie in a
quadratic field.  Both real and imaginary quadratic fields can occur,
but only real fields for forms with trivial character, since the
principal Hecke operators are all Hermitian in the context
of~\cite{JBthesis} where the class number is~$2$.

Many more examples for imaginary quadratic fields with class group of
exponent~$2$ (with class number up to~$8$) have since been computed by
the author, and may be found in the LMFDB~\cite{lmfdb}.

\subsection{Eigenvectors}
Let~$V=\C\M$, where either~$\M=\M_1(\n)$ or~$\M=\M_0(\n)$.  Then $V$
is a finite direct sum~$V=\oplus_c\C\M^{(c)}$, and the action of~$\T$
on~$V$ respects this decomposition by \cref{eqn:T-M-action}.
Let~$v\in V$ be a simultaneous eigenvector for all~$T\in\T$, with
eigenvalue system~$\lambda$, and write~$v=\sum_cv_c$ with~$v\in V_c$.
We now show that each component~$v_c$ is nonzero, except in the case
where the eigenvalue system has (nontrivial) inner twists, or
equivalently, when the the subgroup~$H_0(\lambda)$ of the class group
defined above is a proper subgroup.

\begin{lem}
  \label{lem:eigenvector-components}
  Let~$v=\sum_cv_c$ be a simultaneous eigenvector for~$\T$ with
  eigensystem~$\lambda$.  For all classes~$c_0,c$, if~$T\in\T_{c_0}$
  then
  \[
  T(v_{c}) = \lambda(T)v_{c_0c}.
  \]
\end{lem}
\begin{proof}
For all~$T\in\T$ we have
\[
\lambda(T)\sum_cv_c = \lambda(T)v = T(v) = \sum_cT(v_c);
\]
taking~$T\in\T_{c_0}$ and comparing the $c_0c$-components gives the
result.
\end{proof}

\begin{prop}
  \label{prop:eigenvector-components}
  Let~$v=\sum_cv_c$ be a simultaneous eigenvector for~$\T$ with
  eigensystem~$\lambda$.  If one component of~$v$ is zero,
  say~$v_{c_1}=0$, then $v_c=0$ for all classes~$c$ in the
  coset~$H_0(\lambda)c_1$.
\end{prop}
\begin{proof}
  For all classes~$c_0\in H_0$ there is (by definition) an
  operator~$T$ of class~$c_0$ with~$\lambda(T)\not=0$.
  Now~\Cref{lem:eigenvector-components} implies that~$v_{c_0c}=0$
  also.
\end{proof}

If~$\lambda$ has no inner twists, and in particular when the class
number is odd, $H_0(\lambda)$ is the whole class group, giving the
following.
\begin{cor}
  \label{cor:eigenvector-components-nonzero}
  Let~$v=\sum_cv_c$ be a simultaneous eigenvector for~$\T$.  If the
  eigensystem~$\lambda$ has no inner twists, then every
  component~$v_c$ of~$v$ is nonzero.

  In particular, this is always the case when the class number is odd.
\end{cor}

When there is just one inner twist (which in practice is the only
other case, see \Cref{rmk:only-one-inner-twist}), the
index~$[\Cl(K):H_0]=2$.  Then we can either have~$v_c\not=0$ for
all~$c\in H_0$ and~$v_c=0$ for all~$c\not\in H_0$, or vice versa.

Let $v=\sum_cv_c\in V$ and let~$\psi$ be an unramified character,
viewed as a character of~$\Cl(K)$.  We define the \emph{twist of $v$ by
$\psi$} to be
\begin{equation}
  \label{eqn:eigenvector-twist}
v \otimes \psi = \sum_c\psi^{-1}(c)v_c,
\end{equation}
where each $c$-component has been multiplied by~$\psi^{-1}(c)$.

\begin{prop}
  \label{prop:eigenvector-twist}
  Let~$v=\sum_cv_c$ be a simultaneous eigenvector for~$\T$ with
  eigensystem~$\lambda$, and let $\psi$ be an unramified
  character. Then $v\otimes\psi$ is also a simultaneous eigenvector,
  with eigensystem $\lambda\otimes\psi$.
\end{prop}
\begin{proof}
Write $w=v\otimes\psi=\sum_c\psi^{-1}(c)v_c$.  Let $T\in\T_{c_0}$.
Then, using \Cref{lem:eigenvector-components}, we have
\begin{align*}
T(w) = \sum_c\psi^{-1}(c)T(v_c) &= \sum_c\psi^{-1}(c)\lambda(T)v_{c_0c} \\ &=
\sum_c\psi(c_0)\lambda(T)\psi^{-1}(c_0c)v_{c_0c}\\ &=
\psi(c_0)\lambda(T)w.
\end{align*}
Hence $w$ is an eigenvector for all~$T\in\T$, with eigensystem
$\lambda\otimes\psi=(\alpha\psi,\chi\psi^2)$, where
$\lambda=(\alpha,\chi)$.
\end{proof}

\section{Formal modular forms}
\label{sec:formal-modular-forms}
So far, we have defined operators on modular points in~$\M_0(\n)$
or~$\M_1(\n)$, mapping to~$\Q\M_0(\n)$ or~$\Q\M_1(\n)$, or, in the
case of the level-changing operators~$A_{\d}$, from~$\Q\M_0(\n)$
to~$\Q\M_0(\m)$ for some lower level~$\m\mid\n$.  We now consider a
dual perspective: automorphic forms for~$K$ of level~$\Go(\n)$
or~$\Gi(\n)$ may be defined as functions on~$\M_0(\n)$ or~$\M_1(\n)$,
with values in a complex vector space~$V$, satisfying certain
algebraic and analytic conditions.  We do not go into details
concerning these conditions here, except to say that the weight of the
automorphic forms is encoded in an irreducible representation of the
maximal compact subgroup of the infinite component~$\GL(2,K\otimes\R)$
of the adelic space~$\GL(2,\A_K)$, and the possible analytic
conditions include harmonicity at real places and holomorphicity at
complex places.  When $K=\Q$, or when~$K$ is totally real, $V$ is
one-dimensional and the weight is an integer~$k$ (for~$K=\Q$) or a
vector of~$d=[K:\Q]$ integers in general, and one is led to the
definition of classical or Hilbert modular forms on the upper
half-plane~$\Hyp_2$ or the product~$\Hyp_2^d$.  For fields with
complex places, such as imaginary quadratic fields, one must allow $V$
to have larger dimension, essentially because~$\SU(2,\C)$ is not
abelian and has irreducible representations of every dimension, unlike
the abelian group~$\SO(2,\R)$, and one is led to functions defined
on~$\Hyp=\Hyp_2^{r_1}\times\Hyp_3^{r_2}$, the product of $r_1$ copies
of the upper half-plane~$\Hyp_2$ and $r_2$ copies of its
three-dimensional analogue the upper half-space~$\Hyp_3$. For further
details, we refer to Weil's book~\cite{Weil-SLN189} and (with an
emphasis on the imaginary quadratic case) Bygott's thesis
\cite{JBthesis}.

Here, we consider the space~$\SS_V(\M)$ of \emph{all}
functions~$F\colon\M\to V$, where~$\M=\M_1(\n)$ or~$\M_0(\n)$, and~$V$
is some fixed complex vector space.  By linearity we may then
extend~$F$ to a function~$\C\M\to V$.  We call such a function a
\emph{formal modular form} over~$K$ of level~$\n$.  The
space~$\SS_V(\M)$ is infinite-dimensional (unless~$V$ is trivial),
while the subspaces cut out by the standard definitions of automorphic
forms are finite-dimensional; for our purposes, this distinction does
not matter.

We can define an action of the Hecke algebra~$\T$ on~$\SS_V(\M)$ by
setting
\[
    F\left|T\right. = F\circ T
\]
for $T\in\T$.  Explicitly,
\[
  F\left|T_\a\right.((L,L',\beta)) = N(\a)^{-1}
  \sum_{\substack{[M:L]=\a\\ (M,M',\beta)\in\M}}F((M,M',\beta)),
\]
and for~$\a$ coprime to~$\n$,
\[
  F\left|T_{\a,\a}\right.((L,L',\beta)) =
  N(\a)^{-2}F((\a^{-1}L,\a^{-1}L',\beta)).
\]

For $\alpha\in \OK$ coprime to~$\n$,
\[
  F\left|[\alpha]\right.((L,L',\beta)) = F((L,L',\alpha\beta)).
\]
The operators~$[\alpha]$ give an action of the finite abelian
group~$(\OK/\n)^\times$ on~$\SS_V(\M_1(\n))$, which therefore
decomposes as a direct sum
\[
     \SS_{V}(\M_1(\n)) = \bigoplus_{\chi}\SS_{V}(\n,\chi),
\]
where $\chi$ runs over the group of Dirichlet characters modulo~$\n$,
{\it i.e.}, homomorphisms~$(\OK/\n)^\times\to\C^\times$, and
\[
    \SS_{V}(\n,\chi) = \left\{ F\in\SS_{V}(\M_1(\n)):
    F\left|[\alpha]\right. = \chi(\alpha)F\right\}.
\]
Since each~$[\alpha]$ commutes with the operators in~$\T$ it is again
immediate that~$\T$ acts on each subspace~$\SS_{V}(\n,\chi)$.  If
$\chi_0$ denotes the principal character modulo~$\n$ then
$\SS_{V}(\n,\chi_0)$ may be identified with the
space~$\SS_V(\M_0(\n))$, which we denote simple~$\SS_V(\n)$.

\subsection{Eigenfunction decomposition}

We call a nonzero function~$F\in\SS_V(\M)$ an \emph{eigenfunction}
with eigensystem~$\lambda$ if $F\left|T\right.=\lambda(T)F$ for
all~$T\in\T$.  If also $F\left|W_{\q}\right.=\varepsilon(\q)F$ (for
some scalar~$\varepsilon(\q)$) for all~$\q\mid\mid\n$, we say that~$F$
is a \emph{strong eigenfunction}.  Note that, since
$W_{\q}^2=T_{\q,\q}$, we have~$\varepsilon(\q)^2=\chi(\q)$,
where~$\chi$ is the character of~$\lambda$, so that every~$F$ which is
an eigenfunction for all~$T\in\T$ is automatically an eigenfunction
for~$W_{\q}^2$, but not \textit{a priori} of~$W_{\q}$ itself.
However, if the space of functions with the same eigensystem~$\lambda$
as~$F$ is $1$-dimensional, then, since each~$W_{\q}$ commutes with all
the operators in~$\T$, it follows by elementary linear algebra
that~$F$ is also an eigenfunction for all~$W_{\q}$.

Each~$F\in\SS_V(\M)$ consists of a collection~$(F_c)_{c\in\Cl(K)}$ of
functions~$F_c:\M^{(c)}\to V$, one for each ideal class~$c$, so we may
write~$\SS_V(\M) = \oplus_c\SS_V(\M^{(c)})$.  The Hecke action
respects this decomposition, in the sense that~$T\in\T_{c_1}$ maps
$\SS_V(\M^{(c_2)}) \to \SS_V(\M^{(c_1^{-1}c_2)})$.

For~$F=(F_c)_c\in\SS_V(\M)$ and~$\psi$ an unramified character,
we define the \emph{twist of $F$ by $\psi$} to be~$F\otimes\psi$, with
components given by
\begin{equation}
  \label{eqn:eigenfunction-twist}
(F \otimes \psi)_c = \psi(c)F_c.
\end{equation}
Here the $c$-component is multiplied by~$\psi(c)$, unlike
in~(\ref{eqn:eigenvector-twist}), where~$\psi$ was inverted.

The analogues for eigenfunctions of~\Cref{lem:eigenvector-components},
\Cref{prop:eigenvector-components},
\Cref{cor:eigenvector-components-nonzero},
and~\Cref{prop:eigenvector-twist} are proved in the same way, taking
into account the difference between~(\ref{eqn:eigenvector-twist})
and~(\ref{eqn:eigenfunction-twist}).

\begin{lem}
  \label{lem:eigenfunction-components}
  Let~$F=(F_c)_c\in\SS_V(\M)$ be an eigenfunction with eigensystem~$\lambda$.
  For all classes~$c_0,c$, if~$T\in\T_{c_0}$ then
  \[
  F_c\left|T\right. = \lambda(T)F_{c_0^{-1}c}.
  \]
\end{lem}

\begin{prop}
  \label{prop:eigenfunction-components}
  Let~$F=(F_c)_c\in\SS_V(\M)$ be an eigenfunction with
  eigensystem~$\lambda$.  If one component of~$F$ is zero,
  say~$F_{c_1}=0$, then $F_c=0$ for all classes~$c$ in the
  coset~$H_0(\lambda)c_1$.
\end{prop}

\begin{cor}
  \label{cor:eigenfuction-components-nonzero}
  Let~$F=(F_c)_c\in\SS_V(\M)$ be an eigenfunction with eigensystem~$\lambda$,
  and suppose that~$\lambda$ has no inner twist.  Then every
  component~$F_c$ of~$F$ is nonzero.

  In particular, this is always the case when the class number is odd.
\end{cor}

\begin{prop}
  \label{prop:eigenfunction-twist}
  Let~$F=(F_c)_c\in\SS_V(\M)$ be an eigenfunction with eigensystem~$\lambda$.
  Let $\psi$ be an unramified character. Then $F\otimes\psi$ is also an
  eigenfunction, with eigensystem $\lambda\otimes\psi$.
\end{prop}

\subsection{Changing levels of eigensystems}
\label{subsec:level-changing}
Let $\m$ and~$\n$ be two levels (integral ideals) with $\m\mid\n$.
Recall that each ideal divisor $\d\mid \m^{-1}\n$ we have defined an
operator~$A_\d\colon\Q\M_0(\n)$ to $\Q\M_0(\m)$.

We first note~$A_{\d}$ maps eigenvectors at level~$\n$ to eigenvectors
at level~$\m$.

\begin{prop}
  \label{prop:Ad-preserves-eigenvalues}
Suppose that $v\in\C\M_0(\n)$ is a simultaneous eigenvector for all
operators in~$\T$, say $T(v)=\lambda(T)v$.

Let $\m\mid\n$, and~$\d\mid\m^{-1}\n$, and set~$w=A_{\d}(v)
\in\C\M_0(\m)$.  Then for~$T=T_{\a}$, and for~$T=T_{\a,\a}$
with~$\a\in\J_K^{\n}$, $w$ is also an eigenvector for~$T$ with the
same eigenvalue~$\lambda(T)$ as~$v$.
\end{prop}
\begin{proof}
This follows from~\Cref{prop:Ad-commutes-Tp},
\Cref{prop:Ad-commutes-Taa}, and~\Cref{cor:Ad-commutes-Ta}.  For
$T=T_{\a}$ or~$T_{\a,\a}$, we have $T(w) = TA_{\d}(v) = A_{\d}T(v) =
A_{\d}(\lambda(T)v) = \lambda(T)w$.
\end{proof}

Since~$A_{\d}$ maps $\M_0(\n)$ to~$\M_0(\m)$, lowering the level, the
induced dual action on functions raises the level:
\[
F \in \SS_V(\m) \implies F_{\d} := F\left|A_{\d}\right. = F\circ A_{\d} \in \SS_V(\n).
\]
As in~\Cref{prop:Ad-preserves-eigenvalues}, $A_{\d}$ maps eigenforms
to eigenforms with the same eigenvalues for all $T_{\a}$
and~$T_{\a,\a}$ for~$\a\in\J_K^{\n}$:
\begin{prop}
  \label{prop:T-commutes-Ad-function}
  Let $\m\mid\n$, and let~$F\in\SS_V(\m)$ be an eigenfunction, with
  eigensystem~$\lambda$.  For all~$\d\mid\m^{-1}\n$, the function
  $F_{\d}=F\left|A_{\d}\right.$ is an eigenfunction at level~$\n$,
  with the same eigensystem.
\end{prop}

The behaviour of eigenfunctions for Atkin-Lehner operators is given by
the following result, where, as in~\Cref{prop:Ad-commutes-Wq}, we
restrict to the case where~$\m^{-1}\n$ is a prime power.

\begin{prop}
  \label{prop:Ad-commutes-Wq-function}
  Let $\n=\p^{\alpha+\beta}\m = \d\d'\m$, and suppose
  that~$F\in\SS_V(\m)$ is a strong eigenfunction with
  eigensystem~$\lambda$ and character~$\chi$, such that
  $F\left|W_{\q}^{\m}\right. = \varepsilon(\q)F$.  Let
  $F_{\d}=f\left|A_{\d}\right.$.  Then
  \begin{enumerate}
  \item If~$\p\nmid\q$ then $\q\mid\mid\n$ also, and
    \[
     F_{\d}\left|W_{\q}^{\n}\right. = \varepsilon(\q)F_{\d}.
    \]
  \item If~$\p\mid\q$ then $\q' = \p^{\alpha+\beta}\q = \d\d'\q
    \mid\mid \n$, and
    \[
     F_{\d}\left|W_{\q'}^{\n}\right. = \varepsilon(\q)\chi(\d')F_{\d'}.
    \]
  \end{enumerate}
\end{prop}

\begin{rmk}
 In case~(1), $F_\d$ is a strong eigenfunction, and with the same
 eigenvalue~$\varepsilon(\q)$ for~$W_{\q}$ as~$F$. In case~(2), the
 functions~$F_{\d}$ and~$F_{\d'}$ are interchanged, up to a constant
 factor, so (unless~$\d=\d'$) neither is an eigenfunction
 for~$W_{\q'}^{\n}$.  However, in both cases, the subspace
 of~$\SS_V(\n)$ spanned by all the~$F_{\d}$ (that is, by
 $F\left|A_{\p^\gamma}\right.$ for~$0\le\gamma\le\alpha+\beta$) is
 preserved by~$W_{\q'}^{\n}$, as well as by~$\T$.  To obtain strong
 eigenfunctions, we must form suitable linear combinations of the
 pairs~$F_{\d}$, $F_{\d'}$, specifically~$aF_{\d} +
 \varepsilon(\q)\chi(\d)F_{\d'}$ where~$a^2=\chi(\q')$.

 In the special case where~$\chi$ is trivial, we
 have~$\varepsilon(\q)^2=1$, so~$\varepsilon(\q)=\pm1$, the
 Atkin-Lehner operators are involutions, and~$F_{\d}\pm F_{\d'}$ are
 strong eigenfunctions, as in~\cite[Lemma~2.7.1]{JCbook2}.
\end{rmk}

\begin{rmk}
Since $(W_{\q})^2=T_{\q,\q}$ we have $\varepsilon(\q)^2 = \chi(\q)$,
which is consistent with applying (1) twice.  Applying (2) twice, with
both~$\d$ and~$\d'$, we see that~$F_{\d}$ is an eigenfunction
for~$(W_{\q'})^2$ with eigenvalue~$\varepsilon(\q)^2\chi(\d)\chi(\d')
= \chi(\q')$, which is consistent with $(W_{\q'})^2=T_{\q',\q'}$.
\end{rmk}

For a fixed eigenfunction~$F\in\SS_V(\m)$, and~$\n$ with~$\m\mid\n$,
following classical terminology we may call the subspace
of~$\SS_V(\n)$ spanned by the set of $F\left|A_{\d}\right.$ for
all~$\d\mid\m^{-1}\n$ an \emph{oldspace} at level~$\n$.  The preceding
propositions show that all functions in the oldspace are
eigenfunctions, with the same eigenvalues as~$F$, and that the
oldspace has a basis of strong eigenfunctions.

To define the \emph{newspace} at level $\n$, one would take the
orthogonal complement in~$S_V(\n)$ of all the oldspaces, as~$\m$
ranges over all proper divisors of~$\n$, with respect to a suitable
inner product generalising the Petersson inner product.  Our
space~$S_V(\n)$ is too big for this; but after cutting down to the
subspaces of functions satisfying all the conditions for an
automorphic form, which are finite-dimensional, one is able to define
the inner product via a suitable integral.  We do not elaborate on
this here.
\begin{rmk}
The dual Hecke operators defined above in
section~\ref{subsec:alt-norm} are in fact adjoints of the Hecke
operators with respect to the Petersson inner product, where this can
be defined.
\end{rmk}

\section{Matrices for principal Hecke operators} \label{sec:hecke-matrices}

Classically, Hecke operators may be defined in terms of certain
matrices in $\Mat_2(\Z)$.  This description is particularly useful for
explicit computations.  A similar description for general number
fields~$K$ using matrices in~$\Mat_2(\OK)$ is only possible for
principal operators, essentially because only these map modular points
to modular points whose underlying lattice has the same Steinitz
class.  On the other hand, we saw in~\Cref{sec:Hecke} that the
eigenvalue system for a formal Hecke eigenform is determined up to
unramified twist by the eigenvalues of principal operators.  Hence, in
practice, it suffices to know how to compute the action of
$T_{\a,\a}T_\b$ on $\M_1^{(1)}(\n)$, or $T_{\a,\a}T_{\b}W_{\q}$ on
$\M_0^{(1)}(\n)$, when $\a^2\b$, or $\a^2\b\q$ respectively, is
principal.

It is more convenient to use the dual operators introduced in
subsection~\ref{subsec:alt-norm}, as these may be expressed using
integral matrices; for example, the operator
$\widetilde{T}_{\a,\a}\widetilde T_\b\widetilde{W}_{\q}$ maps a
lattice~$L$ to a formal linear combination of sublattices of
index~$\a^2\b\q$, and may be expressed as a formal linear combination
of matrices in~$\Mat_2(\OK)$, all of determinant~$\delta$,
where~$\delta\in\OK$ generates the principal ideal~$\a^2\b\q$.

Implementations of all these principal operators over imaginary
quadratic fields, may be found in the author's {\tt C++}
package~\cite{bianchi-progs}.

We first discuss the Hecke operators~$\widetilde{T}_{\a,\a}$
and~$\widetilde T_\b$, before turning our attention to Atkin-Lehner
operators.

\subsection{Matrices for Hecke operators}
\subsubsection{Action on~$\RR$}
First consider~$\widetilde{T}_{\b}$ where $\b$ is principal.  To
determine its action on principal lattices, and $L=\RR$ in particular,
we first determine the sublattices~$M\subset\RR$ with $[\RR:M]=\b$.
For each such~$M$, there is a basis~$x,y$ for $\RR$, and uniquely
determined ideals~$\b_1,\b_2$ with $\b=\b_1\b_2^2$, such that
$M=\b_1\b_2x\oplus\b_2y=\b_2(\b_1x+\OK y)$.  If $\gamma\in\Gamma$ is
the matrix with rows~$x,y$ then we have
$M=(\b_1\b_2\oplus\b_2)\gamma$.  Writing $\b_1\b_2\oplus\b_2 =
(\RR)B$, with $B$ a fixed $(\b_1\b_2,\b_2)$-matrix, we have
$M=(\RR)B\gamma$.  As $\gamma$ varies in~$\Gamma$, this gives all
sublattices~$M$ with index~$\b$ and quotient~$\RR/M\cong
\OK/\b_1\b_2\oplus \OK/\b_2$.

This description makes it clear that~$\Gamma$ acts transitively on the
set of lattices with elementary divisors~$\b_1\b_2,\b_2$.  Moreover,
$\gamma_1,\gamma_2\in\Gamma$ give the same lattice~$M$ if and only if
\[
   \Go(\b_1)\gamma_1 = \Go(\b_1)\gamma_2
\]
since $\Gamma\cap\Delta(\b_1\b_2,\b_2)=\Go(\b_1)$.  Hence, to obtain
each such~$M$ exactly once, we let~$\gamma$ run through a complete set
of $\psi(\b_1)$ right cosets of~$\Go(\b_1)$ in~$\Gamma$.  Thus, for
each factorization~$\b=\b_1\b_2^2$, we have $\psi(\b_1)$ matrices~$BC$
where $B$ is a fixed $(\b_1\b_2,\b_2)$-matrix, and~$C$ runs through
coset representatives of~$\Go(\b_1)$ in~$\Gamma$. These cosets are in
bijection with $\P^1(\b_1)$, and we may construct a set of coset
representatives by lifting M-symbols from~$\P^1(\b_1)$ to~$\Gamma$, as
in subsection~\ref{sec:M-symbols}.  The operator~$\widetilde T_{\b}$
is then given by the formal sum of all these, as $\b_2$ runs over the
ideals whose square divides~$\b$.

In the special case where~$\b=\<\beta>$ is square-free, we may take
$B=\diag(\beta,1)$. Then $\widetilde T_{\b}$ is given by the formal
sum of matrices~$BC$ where $C$ runs over coset representatives
of~$\Go(\b)$ in~$\Gamma$, the number of which is
$\psi(\b)=\prod_{\p\mid\b}(1+N(\p))$.  In general, the number of
sublattices of index~$\b=\prod_{\p}\p^e$ is given (after a
straightforward calculation) by~$\eta(\b)$, where
\begin{equation}
  \label{eqn:eta-def}
   \eta(\b) = \prod_{\p\mid\b}(1+N(\p)+N(\p)^2+\dots+N(\p)^e) =
   \prod_{\p\mid\b}
\left(
\frac{N(\p)^{e+1}-1}{N(\p)-1}
\right).
\end{equation}

The operator~$\widetilde{T}_{\a,\a}$, where $\a$ is integral with
$\a^2$ principal, is represented by an~$(\a,\a)$-matrix.  For the
general case of $\widetilde{T}_{\a,\a}\widetilde{T}_\b$ where~$\a^2\b$
is principal, we replace the $(\b_1\b_2,\b_2)$-matrix~$B$ in the above
discussion of~$\widetilde{T}_\b$ with an $(\a\b_1\b_2,\a\b_2)$-matrix.

In summary:
\begin{lem}
\label{lem:index-b-matrices}
  Let~$\a,\b$ be integral ideals with~$\a^2\b=\<\beta>$
  principal. Then there is a set of~$\eta(\b)$ matrices~$g_i$, all
  with determinant~$\beta$, such that
  \[
  \widetilde{T}_{\a,\a}\widetilde{T}_\b(\RR) = \sum_{i=1}^{\eta(\b)}(\RR)g_i.
  \]
  The number of~$g_i$ is~$\eta(\b)=\sum_{\b=\b_1\b_2^2}\psi(\b_1)$,
  where the sum is over all factorizations of~$\b$ of the
  form~$\b=\b_1\b_2^2$. For each such factorization, there
  are~$\psi(\b_1)$ matrices~$g_i=BC$ where~$B$ is a
  fixed~$(\a\b_1\b_2,\a\b_2)$-matrix with~$\det(B)=\beta$ and $C$ runs
  through a set of coset representatives with determinant~$1$
  of~$\Go(\b_1)$ in~$\Gamma$.

  The collection of matrices~$\{g_i\}$ normalises~$\Gamma$, in the
  sense that there is a permutation~$\sigma$
  of~$\{1,2,\dots,\eta(\b)\}$ such that $g_i\Gamma=\Gamma
  g_{\sigma(j)}$ for all~$i$.
\end{lem}
\begin{proof}
Only the last part remains to be proved, and this follows immediately
from the fact that right multiplication by~$\Gamma$ permutes the
lattices of fixed index~$\b$.
\end{proof}

\subsubsection{Action on general modular points}

We have seen how to construct, for all principal ideals~$\b$, a set of
$\eta(\b)$ matrices~$g_i$ such that the sublattices of~$\RR$ of
index~$\b$ are precisely the $(\RR)g_i$ for~$1\le i\le\eta(\b)$.  We
call such a set of matrices \emph{Hecke matrices (for index~$\b$)}.

If $L$ is an arbitrary free lattice, then $L=(\RR)U$ for some
$U\in\G$, and the sublattices $M\subseteq L$ with $[L:M]=\b$ are given
by $M_i=(\RR)g_iU$ for $1\le i\le\eta(\b)$, since
$MU^{-1}\subseteq\RR$, with the same index.  Here, the matrix~$U$ is
only determined up to left multiplication by a
matrix~$\gamma\in\Gamma$.  Replacing~$M_i$ by~$(\RR)g_i\gamma U$ for
each~$i$ gives the same sublattices (possibly permuted), by the last
statement of Lemma~\ref{lem:index-b-matrices}, so the Hecke action on
free lattices is well-defined in terms of the Hecke matrices~$g_i$ by
\[
   \widetilde{T}_{\b}: L=(\RR)U \mapsto \sum_{i=1}^{\eta(\b)}(\RR)g_iU.
\]

Extending this to an action on principal modular points, we must also
take into account the level structure.  We restrict to principal
ideals~$\b$ which are coprime to the level~$\n$.  In the construction
of the $g_i$ we must ensure that each matrix has its~$(2,1)$-entry
in~$\n$.  To do this, we first use $(\b_1\b_2,\b_2)$-matrices~$B$ of
level~$\n$, as defined in subsection~\ref{sec:ab-mats}.  Secondly, the
coset representatives for~$\Go(\b_1)$ in~$\Gamma$ we use must lie
in~$\Gon$.  These may be obtained by lifting from~$\P^1(\b_1)$
to~$\Gamma$ via~$\P^1(\b_1\n)$, as explained in
subsection~\ref{sec:M-symbols}.

When the Hecke matrices~$g_i$ have this additional property, we call
them \emph{Hecke matrices of level~$\n$}.  These now give a
well-defined action on~$\Gon$-modular points:
\[
   \widetilde{T}_{\b}: (L,L')=(\RR,\OK\oplus\n^{-1})U \mapsto
   \sum_{i=1}^{\eta(\b)}(\RR,\OK\oplus\n^{-1})g_iU.
\]
Similarly, for the operator $\widetilde{T}_{\a,\a}$ when $\a$ is
coprime to~$\n$ and $\a^2$ is principal, we must use
an~$(\a,\a)$-matrix of level~$\n$, and in the construction of the
matrices for $\widetilde T_{\a,\a}\widetilde{T}_{\b}$, the
$(\a\b_1\b_2,\a\b_2)$-matrices we use must be of level~$\n$.

\subsection{Explicit special cases}
We now make the constructions above completely explicit in some cases
of particular use in computations, namely
\begin{itemize}
\item $\widetilde{T}_{\p}$ for~$\p$ a principal prime;
\item $\widetilde{T}_{\p^2}$ for~$\p$ prime, with $\p^2$ principal;
  \item $\widetilde{T}_{\p\q}$ for~$\p$, $\q$ distinct primes with~$\p\q$ principal;
\end{itemize}
together with extended version of these where the ideals concerned are
not principal, but have square ideal class, so that we obtain
principal operators by composing with $\widetilde{T}_{\a,\a}$ for
suitable~$\a$.  Here we assume that the level~$\n$ is fixed, and all
these ideals ($\p$, $\q$, and~$\a$) are coprime to~$\n$.

\subsubsection*{Principal primes} Taking
$\b=\p$ to be a principal prime ideal, say $\p=\<\pi>$.  Then there is
only one factorisation of the form~$\b=\b_1\b_2^2$, namely~$(\b_1,
\b_2) = (\p, \OK)$, so we may take $B=\mat{\pi}{0}{0}{1}$. As coset
representatives for~$\Go(\p)$ we take~$\mat{1}{0}{0}{1}$ and
$\mat{0}{1}{1}{x}$ for~$x\pmod{\p}$.  Then the matrices giving the
sublattices of index~$\p$ are
\[
\mat{\pi}{0}{0}{1}, \qquad\text{and}\qquad
\mat{\pi}{0}{0}{1}\mat{0}{1}{1}{x}=\mat{0}{\pi}{1}{x}
\qquad\text{(for~$x\in\OK\pmod{\p}$)}.
\]
As the last~$N(\p)$ of these do not have level~$\n$, we adjust them by
multiplying each on the left by $\mat{0}{1}{1}{0}$ to obtain the
familiar set of Hecke matrices for a principal prime:
\[
  \mat{\pi}{0}{0}{1},\quad\text{and}\quad
  \mat{1}{x}{0}{\pi}\quad\text{for~$x\in\OK\mod{\p}$}.
\]
These have level~$\n$ for every~$\n$.

\subsubsection*{Primes with square class}
Suppose that the ideal class~$[\p]$ is a square, with $\a^2\p$
principal and~$\a$ coprime to~$\n$.  Let $B$ be an~$(\a\p,\a)$-matrix
of level~$\n$, and let $\nu\in\n\setminus\p$ (noting that
$\p\nmid\n\implies\p\not\supseteq\n$).  Then a suitable set of
$N(\p)+1$ matrices
representing~$\widetilde{T}_{\a,\a}\widetilde{T}_\p$ consists of
\[
B, \quad\text{and}\quad B\mat{1}{x}{\nu}{1+x\nu} \quad\text{for
  $x\pmod{\p}$},
\]
since it is easy to see that the matrices~$\mat{1}{x}{\nu}{1+x\nu}$
for~$x\pmod{\p}$ together with the identity matrix represent all the
cosets of $\Go(\p)$, and lie in~$\Go(\n)$.

\subsubsection*{Principal prime squares} Let
$\b=\p^2=\<\beta>$ where $\p$ is a prime whose class has order~$2$.
There are two factorizations $\b=\b_1\b_2^2$ to be considered, namely
$(\b_1, \b_2) = (\OK, \p)$ and $(\b_1, \b_2) = (\p^2, \OK)$.

Taking $\b_2=\p$ gives the lattice $M=\p\oplus\p=(\RR)B$, where~$B$ is
a $(\p,\p)$ matrix.  Hence, for the first Hecke matrix,~$B_1$ can be any
$(\p,\p)$ matrix of level~$\n$.

When $\b_1=\p^2$ and $\b_2=\OK$, we set $B_2=\mat{\beta}{0}{0}{1}$,
and use the matrices $B_2C$ as~$C$ runs through lifts
to~$\Go(\n)$ of all~$(c:d)\in\P^1(\p^2)$.

To make this more explicit, fix~$\nu\in\n\setminus\p$, and first take,
as coset representatives for $\Go(\p^2)$, matrices $\mat{0}{1}{1}{x}$
for~$x\in\OK\pmod{\p^2}$, together with matrices $\mat{1}{0}{y\nu}{1}$
for~$y\in\p\pmod{\p^2}$.  The additional $N(\p)^2+N(\p)$ Hecke
matrices are then:
\begin{align*}
    \mat{1}{x}{0}{\beta}& \qquad\text{for
      $x\in\OK\mod{\p^2}$};\\
    \mat{\beta}{0}{y\nu}{1}& \qquad\text{for
      $y\in\p\mod{\p^2}$};
\end{align*}
here, for the first set, we have again swapped the rows to obtain
matrices of level~$\n$.

A similar set of Hecke matrices were used in Bygott's
thesis~\cite[Prop.~123]{JBthesis}; he specified that
$\nu\equiv1\pmod\p$, which is stronger than necessary.

\subsubsection*{$\widetilde{T}_{\a,\a}\widetilde{T}(\p^2)$ for arbitrary
  primes, with $\a\p$ principal} For any prime~$\p$ not dividing~$\n$,
we may take~$\a$ to be an ideal coprime to~$\n$ in the inverse class
to~$\p$, so that the operator~$\widetilde{T}_{\a,\a}\widetilde{T}(\p^2)$ is
principal.  As before, this operator is a formal sum of
$N(\p)^2+N(\p)+1$ matrices.  Taking $B_1$ to be an $(\a\p,\a\p)$-matrix
of level~$\n$ and $B_2$ to be an $(\a\p^2,\a)$-matrix of level~$\n$,
we use $B_1$ and $B_2C$ as $C$ again runs through lifts
to~$\Go(\n)$ of all $(c:d)\in\P^1(\p^2)$.

\subsubsection*{Principal products of two primes} Let
$\b=\p\q=\<\beta>$ where $\p$ and~$\q$ are distinct prime ideals in
inverse ideal classes.  There is only one factorization of~$\b$ to be
considered, namely~$(\b_1, \b_2) = (\p\q,\OK)$.  Using similar
arguments as before, we find the following set
of~$\eta(\p\q)=(N(\p)+1)(N(\q)+1)$ Hecke matrices of level~$\n$
(assuming that $\p\q$ is coprime to~$\n$), taking
$\nu\in\n\setminus(\p\cup\q)$:
\begin{align*}
  \mat{1}{x}{0}{\beta}&\qquad\text{for $x\mod{\p\q}$;}\\
  \mat{\beta}{0}{y\nu}{1}&\qquad\text{for $y\in\p\cup\q\mod{\p\q}$;}
\end{align*}
together with both {a $(\p,\q)$-matrix of level~$\n$} and {a
  $(\q,\p)$-matrix of level~$\n$}.  There are $N(\p)N(\q)$ matrices of
the first kind and~$N(\p)+N(\q)-1$ of the second kind.  Again, these
are similar to Hecke matrices used by Bygott
in~\cite[Prop.~124]{JBthesis}.

\subsubsection*{Products of two primes whose product has square class}
More generally, suppose that $\p$ and~$\q$ are distinct primes such
that the class of~$\p\q$ is a square, with $\a^2\p\q$ principal (all
ideals coprime to~$\n$). Then the principal
operator~$\widetilde{T}_{\a,\a} \widetilde{T}_{\p\q}$ is again a sum of
$(N(\p)+1)(N(\q)+1)$ matrices, namely $BC$ where $B$ is a
fixed~$(\a\p\q,\a)$-matrix of level~$\n$ and~$C$ runs through lifts
to~$\Go(\n)$ of all~$(c:d)\in\P^1(\p\q)$.

\subsection{Matrices for Atkin-Lehner operators}
\label{sec:Atkin-Lehner-matrices}
The material in this section extends the special cases treated in
Bygott's thesis (\cite[Section 1.4]{JBthesis}) and Lingham's thesis
(\cite[Section 5.3]{LinghamThesis}), and first appeared in the thesis
of Aran\'es (\cite[\S2.3.1]{MaiteThesis}). Bygott used an overgroup of
$\Gon$, containing $\Gon$ as a normal subgroup of finite index,
constructed from ideals whose class has order~$2$ in the class group.
In our notation, this overgroup consists of all $(\a,\a)$-matrices
where $\a$ is an integral ideal whose square is principal.  When the
class number is odd, a method of generalizing the classical
construction of Atkin-Lehner operators was given by Lingham.

Here we treat the following cases, where $\n$ is the level and
$\q\mid\mid\n$:
\begin{itemize}
\item the Atkin-Lehner operator~$W_\q$, when $\q$ is principal;
\item the operator $T_{\m,\m}W_\q$ with~$\m$ coprime to~$\n$, when
  $\m^2\q$ is principal;
\item The operator~$T_{\p}W_\q$ with $\p$ a prime not dividing~$\n$,
  when $\p\q$ is principal.
\end{itemize}
The first two of these can each be defined using a single matrix
in~$\Mat_2(\OK)$, whose determinant generates~$\q$ or $\m^2\q$
respectively, and we deal with these first.

As in the previous section, it is more convenient to use the dual
operators~$\widetilde{W}_\q$, $\widetilde{T}_{\m,\m}\widetilde{W}(\q)$, and
$\widetilde{T}_{\p}\widetilde{W}(\q)$, since these are realised by integral
matrices.

\subsubsection{$W_{\q}$ with $\q$ principal}
By definition, the operator~$\widetilde{W}_{\q}$ maps the standard modular
point~$(\OK\oplus\OK,\OK\oplus\n^{-1}\OK)$ to~$(\q\oplus\OK,
\OK\oplus\q\n^{-1})$.

\begin{defn}
Let $\q$ be a principal ideal such that $\q\mid\mid\n$.  A
\emph{$W_\q$-matrix of level~$\n$} is a matrix~$M=\mat{x}{y}{z}{w}$
such that
\[
  x,w\in\q;\qquad y\in\OK;\qquad z\in\n;\qquad \<\det M>=\q.
\]
More concisely,
\[
   M \in \mat{\q}{\OK}{\n}{\q} \qquad\text{with}\quad \<\det M>=\q.
\]
\end{defn}

$W_{\q}$-matrices are exactly the matrices~$M$ such that
\begin{align*}
   (\OK\oplus \OK,\OK\oplus\n^{-1})M &= (\q\oplus \OK,\OK\oplus\q\n^{-1});\\
   (\q\oplus \OK,\OK\oplus\q\n^{-1})M &= \q(\OK\oplus \OK,\OK\oplus\n^{-1}).
\end{align*}
Equivalently,
\begin{align*}
  \mat{\OK}{\OK}{\n}{\OK}M &= \mat{\q}{\OK}{\n}{\q}, \quad\text{and}\\
  \mat{\q}{\OK}{\n}{\q}M &= \q\mat{\OK}{\OK}{\n}{\OK}.
\end{align*}

$W_\q$-matrices exist, and may be constructed as follows:

\begin{prop}[Existence of $W_\q$-matrices] \label{prop:Wq-exist}
$W_\q$-matrices exists for every principal exact
  divisor~$\q\mid\mid\n$.
\end{prop}
\begin{proof}
let $\q=\<g>$, let $\a\in[\n]^{-1}$ be an integral ideal coprime
to~$\n$, with $\a\n=\<z>$, and let~$\b\in[\q]^{-1}$ be an integral
ideal coprime to~$\a\q'$, with $\b\q=\<x>$.  Then $\b\q$ and~$\a\q'$
are coprime, so $\q=\b\q^2+\a\n=\<xg>+\<z>$, so we can write
$g=gxw-zy$ with~$y,w\in\OK$.  Now $M=\mat{x}{y}{z}{gw}$ has the
desired properties.
\end{proof}

For further properties of~$W_{\q}$-matrices, see the next subsection,
taking~$\m=\OK$.

\subsubsection{$T_{\m,\m}W_{\q}$ with $\m^2\q$ principal}
With~$\q$ still an exact divisor of~$\n$, we now only assume that the
ideal class $[\q]$ is a square, and let~$\m$ be an ideal coprime
to~$\q'$ such that $\q\m^2$ is principal.
\begin{defn}
A $W_\q^\m$-matrix of level~$\n$ is a matrix~$M=\mat{x}{y}{z}{w}$ such
that
\[
  x,w\in\m\q;\qquad y\in\m;\qquad z\in\m\n;\qquad \<\det M>=\q\m^2.
\]
More concisely,
\[
   M \in \mat{\m\q}{\m}{\m\n}{\m\q} \qquad\text{with}\quad \<\det
   M>=\q\m^2.
\]
\end{defn}

$W_{\q}^{\m}$-matrices are exactly the matrices~$M$ such that
\begin{align*}
   (\OK\oplus \OK,\OK\oplus\n^{-1})M &= \m(\q\oplus \OK,\OK\oplus\q\n^{-1}),\quad\text{and}\\
   (\q\oplus \OK,\OK\oplus\q\n^{-1})M &= \m\q(\OK\oplus \OK,\OK\oplus\n^{-1}),
\end{align*}
or alternatively,
\begin{align*}
  \mat{\OK}{\OK}{\n}{\OK}M &= \mat{\m\q}{\m}{\m\n}{\m\q}=\m\mat{\q}{\OK}{\n}{\q},\quad\text{and}\\
  \mat{\q}{\OK}{\n}{\q}M &= \mat{\m\q}{\m\q}{\m\n\q}{\m\q}=\m\q\mat{\OK}{\OK}{\n}{\OK}.
\end{align*}

When $\q=\m=\<1>$, these matrices are simply elements of~$\Gon$.  When
$\n=\q=\<1>$, these matrices, for all~$\m$ such that $\m^2$ is
principal, generate the normalizer of $\Gamma$ in $\GL(2,K)$ (modulo
scalars), denoted~$\Delta$ in~\cite{JBthesis}.  When $\q=\<1>$ and
$\m^2$ is a principal ideal coprime to~$\n$, then $W_\q^\m$-matrices
of level~$\n$ give matrix representations of the Hecke
operators~$\widetilde{T}_{\m,\m}$, as in the previous section, while
if~$\m=\OK$ then a $W_{\q}^{\m}$-matrix is just a $W_{\q}$-matrix.

Note that if we replace $\m$ by an alternative ideal~$\m'$ in the same
class, then $W_\q^\m$-matrices and $W_\q^{\m'}$-matrices only differ
by a scalar factor.  If we identify two such matrices when they differ
by a scalar factor, then $W_\q^\m$-matrices are associated with pairs
$(\q,[\m])$ such that $[\q\m^2]=0$, and the number of such pairs is
finite for each level~$\n$.

A $W_\q^{\m}$-matrix may also be described as a~$(\m\q,\m)$-matrix of
level~$\n$, whose adjugate is also an~$(\m\q,\m)$-matrix.  In case
$\q=\<1>$, a $W_{\<1>}^{\m}$-matrix is just an~$(\m,\m)$-matrix of
level~$\n$ (where $\m^2$ is principal).

$W_\q^{\m}$-matrices may be constructed in a similar way to
$W_\q$-matrices:

\begin{prop}[Existence of $W_\q^\m$-matrices] \label{prop:Wmq-exist}
$W_\q^\m$-matrices exists for every exact divisor $\q\mid\mid\n$
  and~$\m$ such that $\q\m^2$ is principal.
\end{prop}
\begin{proof}
Let $\q\m^2=\<g>$, let $\a\in[\m\n]^{-1}$ be an integral ideal coprime
to $\m\n$, with $\a\m\n=\<z>$, and let $\b\in[\m\q]^{-1}$ be an
integral ideal coprime to $\a\q'$, with $\b\m\q=\<x>$.

Then
\[
  \<gx,z> = \q\m^2\b\m\q+\a\m\n = \m\q(\b\m^2\q+\a\q') = \m\q,
\]
since each of $\b$, $\m$, $\q$ is coprime to each of~$\a$, $\q'$.
Hence $g\in\m^2\q=\m\<gx,z>$, so $g=gxw-zy$ with $y,w\in\m$.  Now
$\mat{x}{y}{z}{gw}$ is an $W_\q^\m$-matrix.
\end{proof}

\begin{prop}[Uniqueness of $W_\q^\m$-matrices]\label{prop:Wmq-equiv}
Let $M$, $M_1$, and~$M_2$ be $W_\q^\m$-matrices of level~$\n$ (with
the same~$\n$, $\q$, and~$\m$).  Then
\begin{enumerate}
\item $M_1M_2\in\Gon$ (up to a scalar factor).
\item $M_1M_2^{-1}\in\Gon$ and $M_1^{-1}M_2\in\Gon$.
\item The set of all $W_\q^\m$-matrices of level~$\n$ equals the left
  coset $M\Gon$ and also the right coset $\Gon M$.
\item If $\m_1$ is another ideal such that $\q\m_1^2$ is principal,
  say $\m_1=\m\b$ with $\b^2$ principal, then the set of all
  $W_\q^{\m_1}$-matrices is
\[
  \{ MB \mid B\quad\text{is a $(\b,\b)$-matrix of level~$\n$}\}.
\]
\end{enumerate}
\end{prop}
\begin{proof}
Part (1) is a special case of the more general result
in~\Cref{prop:Wmq-prod} below (recalling that $\Gon$ is the set of
$W_{\<1>}^{\<1>}$-matrices).  Part~(2) follows from this, since up to
scalars we have $M_i^{-1}=\adj M_i$, another~$W_\q^\m$-matrix.  This
already shows that every $W_\q^\m$-matrix is in both $M\Gon$ and $\Gon
M$; that these both consist only of $W_\q^\m$-matrices also follows
from~\Cref{prop:Wmq-prod}.  Part~(4) is a similar easy calculation,
using the fact that when $\b^2$ is principal, a $(\b,\b)$ matrix of
level~$\n$ is the same as a $W_{\<1>}^\b$-matrix.
\end{proof}

\begin{prop}[Products of $W_\q^\m$-matrices] \label{prop:Wmq-prod}
For $i=1,2$ let $M_i$ be a $W_{\q_i}^{\m_i}$-matrix of level~$\n$,
where $\q_i\mid\mid\n$ and $\m_i^2\q_i$ is principal.  Then
$M_3=M_1M_2$ is (up to a scalar factor) a $W_{\q_3}^{\m_3}$-matrix of
level~$\n$, where $\m_3=\a\m_1\m_2$, $\a=\q_1+\q_2$ and
$\q_3=\q_1\q_2\a^{-2}$.
\end{prop}
\begin{proof}
If we write $\q_1=\a\q_3'$ and $\q_2=\a\q_3''$ then $\q_3=\q_3'\q_3''$
and $\q_3$ is another exact divisor of~$\n$.  We
also have $\q_3\m_3^2 = \q_3(\a\m_1\m_2)^2 =
(\q_1\m_1^2)(\q_2\m_2^2)$, which is principal, showing that $M_3$ has
the right determinant to be a $W_{\q_3}^{\m_3}$-matrix of level~$\n$.
Finally,
\begin{align*}
M_3 \in
\mat{\m_1\q_1}{\m_1}{\m_1\n}{\m_1\q_1}\mat{\m_2\q_2}{\m_2}{\m_2\n}{\m_2\q_2}
&\subseteq
\mat{\m_1\m_2(\q_1\q_2+\n)}{\m_1\m_2(\q_1+\q_2)}{\m_1\m_2\n(\q_1+\q_2)}{\m_1\m_2(\q_1\q_2+\n)}
\\ &= \mat{\m_3\q_3}{\m_3}{\m_3\n}{\m_3\q_3}
\end{align*}
since $\q_1+\q_2=\a$ and $\q_1\q_2+\n=\a\q_3$.
\end{proof}

\begin{cor}
  \label{cor:def-Gon-tilde}
  For a fixed level~$\n$, the set of all $W_\q^{\m}$-matrices of
  level~$\n$, as $\q$ ranges over all exact principal divisors of~$\n$,
  form a group~$\widetilde{\Gon}$ under multiplication modulo scalar
  matrices.  This group contains $\Gon$ (modulo scalars) as a normal
  subgroup of finite index, such that the quotient is an elementary
  $2$-group.\qed
\end{cor}

Since $\Gon$-modular points are right invariant under multiplication
by~$\Gon$, and using~\Cref{prop:Wmq-equiv}, we have a well defined
action on the set $\M_0^{(1)}(\n)$ of all principal $\Gon$-modular
points by mapping
\[
   P=(\OK\oplus \OK,\OK\oplus\n^{-1})U \mapsto (\OK\oplus \OK,\OK\oplus\n^{-1})MU;
\]
this is well defined, as the image is unchanged if we either
replace~$U$ by another admissible basis matrix~$\gamma U$ or
replace~$M$ by another~$W_\q^\m$-matrix~$\gamma M$ for
any~$\gamma\in\Gon$.

As already mentioned, in the special case~$\q=\<1>$, we recover the
action of $\widetilde{T}_{\m,\m}$ for~$[\m]\in\Cl[2]$.  When $\m=\<1>$
and~$\q$ is principal, we obtain the Atkin-Lehner operator~$W_\q$.

\subsubsection{$T_{\a}W_{\q}$ with $\a\q$ principal}

In applications, it is convenient to be able to apply the operator
$\widetilde{T}_{\a}\widetilde{W}_{\q}$ at level~$\n$, where
$\q\mid\mid\n$ and $\a$ is coprime to~$\n$, with $\q\a$ principal.
Here, for simplicity, we restrict to the case where $\a=\p$ is prime.
Now, $\widetilde{T}_{\p}\widetilde{W}_{\q}$ will be the formal sum of
$\psi(\p)=N(\p)+1$ matrices~$M$ with entries in~$\OK$, satisfying the
following properties, where $\<g>=\p\q$ and $L$ runs through all
$\psi(\p)$ lattices of index~$\p$:
\begin{enumerate}
\item $\det(M)=g$;
\item $M\in \mat{\q}{\OK}{\n}{\q}$;
\item $(\RR)M \subseteq L$.
\end{enumerate}
We construct these as follows.  Let $\b\in[\q']^{-1}$ be coprime
to~$\p\q$, so $\b\q'=\<h>$ with $h\in\OK$ coprime to~$\p$.  For each
$(c_0:d_0)\in\P^1(\p)$, use the Chinese remainder Theorem to find
$(c:d)\in\P^1(\p\n\b)$ such that
\[
(c:d) = \begin{cases}
  (c_0:d_0)\quad\text{in $\P^1(\p)$};\\
  (1:0)\qquad\text{in $\P^1(\q)$};\\
  (0:1)\qquad\text{in $\P^1(\b\q')$}
\end{cases}
\]
and lift to a matrix~$\mat{a}{b}{c}{d}\in\Go(\b\q')$. Then set
$M=\mat{d}{c/h}{bgh}{ag}$. Note that $c/h\in\OK$, since
$c\in\b\q'=\<h>$. Properties~(1) and~(2) for~$M$ are immediate, using
$d,g\in\q$ and $h\in\b$.  As for~(3), the entries in the second row
of~$M$ lie in~$\p$, while those in the top row lie
in~$L=\{(x,y)\in\RR\mid c_0x\equiv d_0hy\pmod{\p}\}$, and $L$ ranges
over all the lattices of index~$\p$ as $(c_0:d_0)$ runs
over~$\P^1(\p)$, since~$h$ is coprime to~$\p$.

\subsubsection{Further remarks}

Elements of the quotient group~$\widetilde{\Gon}$ defined
in~\Cref{cor:def-Gon-tilde} induce involution operators on any space
on which $\Gamma$ acts, since they preserve $\Gon$-invariant
subspaces.  These involutions generalize both classical Atkin-Lehner
involutions, and also include the involutions coming from elements of
order~$2$ in the class group as utilized in~\cite{JBthesis}.  The size
of the group of involutions depends on both the structure of the ideal
class group and also the ideal classes of the divisors of~$\n$.  At
one extreme, if the ideal class group is an elementary abelian
$2$-group, then no non-trivial ideal class is a square.  In this case
the only $W_\q^{\m}$-matrices are those for which $\q$ is principal,
but $\m$ is arbitrary.  On the other hand, if the class number is odd
(as in Lingham's thesis \cite{LinghamThesis}) then every ideal class
is a square but only principal ideals have principal squares; in this
case we have $W_\q^{\m}$-matrices for all~$\q\mid\mid\n$ but $\m$
(strictly, its ideal class) is uniquely determined by~$\q$.  (In
\cite{LinghamThesis}, the choice was fixed as $\m=\q^n$ where $2n+1$
is the class number.) This is closest to the classical situation.

The analogue of the classical Fricke involution at level~$\n$
is~$W_{\n}$, which only acts on principal modular points when~$\n$ is
principal, when it is given by a~$W_{\n}^{\<1>}$-matrix of level~$\n$.
If $\n$ has square ideal class, with $\n\m^2$ principal,
then~$W_{\n}\widetilde{T}_{\m,\m}$ acts via a~$W_{\n}^{\m}$-matrix.

\end{document}